\documentclass[a4paper,10pt]{article}
\usepackage{graphicx, color}
\usepackage{amsmath,amssymb,amsthm,comment,cite}
\usepackage{enumitem}
\usepackage{hyperref}
\usepackage{cleveref}

\newcommand{\spex}{{\mathrm spex}}

\newcommand{\conv}{{\mathrm conv}}

\renewcommand{\sup}{\mathrm{supp}}

\newtheorem{thm}{Theorem}[section]
\newtheorem{cor}[thm]{Corollary}
\newtheorem{obs}[thm]{Observation}
\newtheorem{prop}[thm]{Proposition}
\newtheorem{prob}[thm]{Problem}
\newtheorem{lem}[thm]{Lemma}

\newtheorem{quest}[thm]{Question}

\theoremstyle{definition}
\newtheorem{defi}{Definition}[section]

\newcommand{\mc}[1]{\mathcal{#1}}

\crefname{enumi}{Theorem}{Theorems}

\title{Hamiltonian connectivity of some base-cobase graphs}
\author{Leonardo Martínez-Sandoval\thanks{Facultad de Ciencias, Universidad Nacional Autónoma de México, Ciudad de México, México.} \and Kolja Knauer\thanks{Departament de Matem\'{a}tiques i Inform\'{a}tica, Universitat de Barcelona, Centre de Recerca Matemàtica, Barcelona, Spain.}}
\date{}

\begin{document}

\maketitle


\begin{abstract}
 There has been wide interest in understanding which properties of base graphs of matroids extend to base-cobase graphs of matroids. A significant result of Naddef and Pulleyblank (1984) shows that the $1$-skeleton of any $(0,1)$-polytope is either a hypercube, or Hamiltonian connected, i.e. there is a Hamiltonian path connecting any two vertices. In particular, this is true for base graphs of matroids. A natural question raised by Farber, Richter, and Shank (1985) is whether this extends to base-cobase graphs. 

First, we use the polytopal approach to show Hamiltonian connectivity of base-cobase graphs of series-parallel extensions of lattice path matroids. On the other hand, we show that this method extends to only very special classes related to identically self-dual matroids.
Second, we show that base-cobase graphs of wheels and whirls are Hamiltonian connected. 
Last, we show that the regular matroid $R_{10}$ yields a negative answer to the question of Farber, Richter, and Shank. 

\end{abstract}

\section{Introduction}

\subsection{State of the art}
In this section we introduce basic notions and give an overview of the state of the art of the main problems associated to base-cobase graphs of matroids.
We assume the reader to be familiar with basic matroid theory, see Oxley~\cite{Oxley2006}.
A matroid (of rank $r$) is a pair $M=(E,\mathcal{B})$ of a finite ground set $E$ and a non-empty set $\mathcal{B}$ of ($r$-element) subsets called \emph{bases} satisfying the following \emph{basis exchange axiom}:
\begin{center}
    for all $B,B'\in \mathcal{B}$ and $e\in B\setminus B'$ there is $f\in B'\setminus B$ such that $B\setminus\{e\}\cup\{f\}\in\mathcal{B}$.
\end{center}
The \textit{base graph} $G(M)$ of a matroid $M$ has as vertex set the set $\mc{B}$ of bases of $M$ and two bases $B,B'\in\mc{B}$ are adjacent if and only if the symmetric difference $B\Delta B'$ consists of two elements or, equivalently, when $B$ can be obtained from $B'$ by a single application of the basis exchange axiom. It is well-known that $G(M)$ determines $M$ up to isomorphism and duality modulo loops and coloops, see~\cite{HNT73}. Several graph theoretic characterizations of base graphs are available, see~\cite{Che17,CCO15,Mau73}. A connection to discrete and tropical geometry~\cite{Huh18} is granted by the fact that $G(M)$ is the $1$-skeleton of the \emph{matroid base polytope} $$P_M=\conv(x_B\mid B\in\mathcal{B})\subseteq\mathbb{R}^E,$$ where $x_B\in\{0,1\}^E$ denotes the \emph{incidence vector} of the set $B\subseteq E$.

\begin{figure}[ht]
    \centering
    \includegraphics[width=.7\textwidth]{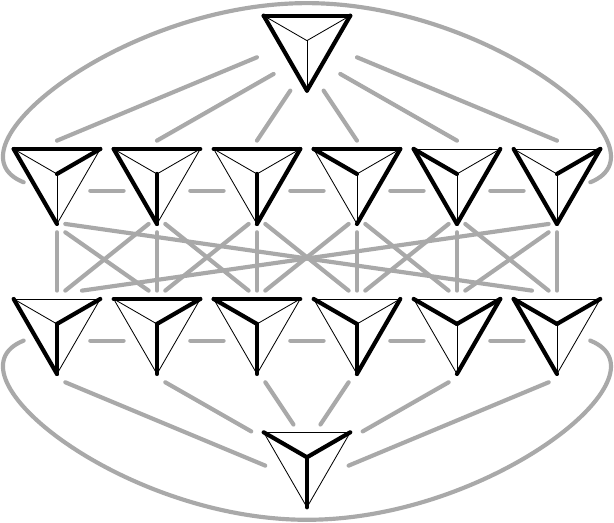}
    \caption{The base-cobase graph of the whirl $\mathcal{W}^3$.}
    \label{fig:lowerbounds}
\end{figure}

If in a matroid $M=(E,\mathcal{B})$ there exists $B\in \mc{B}$ such that also $\overline{B}:=E\setminus B\in \mc{B}$, then we call $M$ a \textit{block matroid} and any such $B$ a \textit{base-cobase}. 
For a block matroid $M$, the \textit{base-cobase} graph $G(M,M^\ast)$ is the subgraph of $G(M)$ induced by the base-cobases of $M$. See \Cref{fig:lowerbounds} for an example. The \emph{base-cobase polytope} $P_{M,M^*}$  of $M$, i.e, the convex hull of  the incidence vectors $x_B\in\{0,1\}^E$ of all $B\in\mathcal{B}\cap\mathcal{B}^*$, has been studied by Cordovil and Moreira~\cite{Cordovil1993}. 

Several properties enjoyed by base graphs, and in part even by skeleta of $(0,1)$-polytopes, have been suspected to extend to base-cobase graphs. 

%


\begin{prob}
  \label{conjbc}
  Let $\mathbf{M}$ be the class of all matroids. Then:

  \begin{itemize}
       \item \textsc{Con}: For every block matroid $M\in \mathbf{M}$, we have that $G(M,M^\ast)$ is connected.
         (Farber, Richter, Shank~\cite[Problem (iii)]{Farber1985})
         \item \textsc{Circ}: For every block matroid $M\in \mathbf{M}$ of rank $r$, there is a base-cobase $B$ of $M$ such that $d_{G(M,M^\ast)} (B,E \setminus B) =r$.
         (Kajitani, Ueno, Miyano~\cite[Conjecture]{Kajitani1988})
         \item \textsc{SCirc}: For every base-cobase $B$ of $M\in\mathbf{M}$ of rank $r$, we have $d_{G(M,M^\ast)} (B,E \setminus B) =r$.
        (Gabow~\cite{Gab76}, Cordovil, Moreira~\cite[Conjecture 1.4]{Cordovil1993}))
        \item \textsc{Diam}:  For every block matroid $M\in \mathbf{M}$ of rank $r$, we have $\text{diam}(G(M,M^\ast))=r$.
        (Hamidoune~\cite[Conjecture 1.5]{Cordovil1993})
        \item \textsc{Poly}: There is a $c>0$ such that $\text{diam}(G(M,M^\ast))=O(r^c)$ for all block matroids $M \in \mathbf{M}$.
        (Andres, Hochst{\"a}ttler, Merkel~\cite[Problem 35]{AHM14})
        \item \textsc{Ham}: For every block matroid  $M\in \mathbf{M}$, we have that $G(M,M^\ast)$ is \emph{Hamiltonian connected}, i.e., for all $u,v$ there exists a Hamiltonian path with endpoints $u,v$, or $G(M,M^\ast)$ is a hypercube. (Farber, Richter, Shank~\cite[Problem (iii)]{Farber1985})
  \end{itemize}
  
\end{prob}

\begin{figure}[ht]
    \centering
    \includegraphics[width=.4\textwidth]{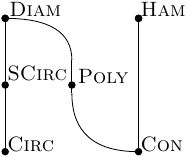}
    \caption{The implications among the properties from~\Cref{conjbc} (the higher, the stronger).}
    \label{diagram}
\end{figure}

Motivations for~\Cref{conjbc} range through  equitability of matroids, fair allocations, toric ideals, and reconfiguration problems, see the introduction of~\cite{brczi2023reconfiguration}. Many researchers have provided partial results on these problems for special classes of matroids. The dominant results are:
\begin{itemize}
    \item \textsc{Con} holds for linear frame matroids (McGuiness~\cite{McG20}). This implies the result for graphic matroids (Farber, Richter, Shank \cite{Farber1985}).
    \item \textsc{Poly} holds for regular matroids (B{\'e}rczi, M{\'a}trav{\"o}lgyi.
    \item \textsc{SCirc} holds for regular matroids (B{\'e}rczi, M{\'a}trav{\"o}lgyi,  Schwarcz~\cite{brczi2023reconfiguration}), which implies \textsc{SCirc} for graphic matroids (Cordovil, Moreira~\cite{Cordovil1993}). This implies \textsc{Circ} for graphic matroids (Kajitani, Ueno, Miyano \cite{Kajitani1988}). 
    \item \textsc{Diam} holds for split matroids (Bérczi, Schwarcz~\cite{Berczi2022}). This implies the result for paving matroids (Bonin~\cite{Bonin2013}). Further, \textsc{Diam} holds for strongly base orderable matroids, wheels, and spikes (B{\'e}rczi, M{\'a}trav{\"o}lgyi,  Schwarcz~\cite{BMS24}). 
    This implies \textsc{Con} for transversal matroids (Farber \cite{Far-89}).
\end{itemize}

The property \textsc{Ham} has not been investigated much in base-cobase graphs, however a lot of research has been dedicated to it in related graph classes, since it may be seen as a problem on \emph{combinatorial Grey codes}~\cite{Mut23}. As mentioned above, Naddef and Pulleyblank~\cite{NO84} show Hamiltonian connectivity for $1$-skeleta of $(0,1)$-polytopes that are not hypercubes. Also see~\cite{MM24} for related algorithmic results. While the hypercube itself is not Hamiltonian connected, by \emph{Havel's lemma} \cite{havel1984Hamiltonian} it is \emph{Hamiltonian laceable}, i.e., every pair of vertices from different color classes is connected by a Hamiltonian path. 
For general subgraphs of the hypercube Hamiltonicity has been an active area of research. Here, probably one of the most noted results is the proof of the middle levels conjecture~\cite{Mut16,Mut24}, which has been extended to the fact that the middle levels graph is Hamiltonian laceable~\cite{GMM23bis} and that the graphs induced by a centrally symmetric interval of levels are Hamiltonian~\cite{GMM23}. 
Further strengthenings of Havel's lemma and extensions to subgraphs of the hypercube with few vertex deletions have been studied by Castañeda and Gotchev~\cite{castaneda2015path}. Their results play a crucial role in some of our proofs.

\subsection{Surrounding classes}
Even though most of our results are about very specific classes of matroids, they have some impact on the understanding of some larger surrounding classes, which we will breifly introduce here.




The class of \textit{lattice path matroids} (\textrm{LPM}) was introduced by Bonin, de Mier, and Noy \cite{Bon2003}. Many different aspects of LPM have been studied: excluded minor characterizations~\cite{Bon-10}, toric ideals~\cite{Sch-11}, the Tutte polynomial~\cite{Bon2003,KMR18,Mor-13}, matroid quotients~\cite{BK22,DeM-07,benedetti2025shellabilityquotientorderlattice}, and the base polytope~\cite{An-17,benedetti2023latticepathmatroidpolytopes,knauer2018lattice}. In \Cref{sec:spex} we introduce the class $\spex(\mathrm{LPM})$ of series-parallel extensions of the class $\textrm{LPM}$. This class strictly contains $\textrm{LPM}$, since it contains all series-parallel graphic matroids, but not all of them are $\mathrm{LPM}$, see~\cite{KMR18}.
Further, $\spex(\mathrm{LPM})$ contains the series-parallel extensions of uniform matroids, a class studied by Chaourar and Oxley~\cite{Chaourar2003}. 

A different generalization of $\mathrm{LPM}$ is the class of \emph{multipath matroids}, introduced by Bonin and Giménez~\cite{Bon-07}. The class of multipath matroids strictly contain $\mathrm{LPM}$, because \emph{whirls} are multipath matroids. This latter class is central to our paper and will be carefully introduced in~\Cref{WW}.

All the above classes are \emph{positroids}, a class independently introduced by da Silva~\cite{dS87}, Blum~\cite{B07}, and  Postnikov~\cite{P06}. Positroids are gammoids~\cite{CH24}, gammoids are strongly base orderable~\cite[Theorem 42.11]{Sch03}, and strongly base orderable matroids satisfy \textsc{Diam} by~\cite{BMS24}. Hence, positroids satisfy \textsc{Diam}.

On the other hand, the class of \emph{regular} matroids comprises those matroids that represent the linear dependencies of a totally unimodular matrix, meaning that every square submatrix has a determinant of $\pm 1$ or $0$, see~\cite{Oxley2006}. A key structural property is Seymour’s decomposition theorem~\cite{Sey80}, which provides a method for decomposing any regular matroid into matroids that are either graphic, cographic, or isomorphic to a simple 10-element matroid called $R_{10}$. This matroid will be discussed in detail in~\Cref{sec:regular}. Its base cobase graph also plays a role with respect to a recent strengthening of \textsc{SCirc},~\cite{GMOSY24}. Graphic and cographic matroids are regular. In particular, \emph{wheels} are regular. This latter class is also central to our paper and will be carefully introduced in~\Cref{WW}.


\subsection{Our contributions}

Since all properties in~\Cref{conjbc} are satisfied by base graphs of matroids we define:

\begin{defi}
    A family $\mathbf{M}$ of matroids has property \textsc{Mat} if for every block matroid  $M\in \mathbf{M}$, we have $G(M,M^\ast)$ is equal to\footnote{This means equality of graphs, not isomorphism. This is equivalent to the set of base-cobases of $M$ being the set of bases of $N$} $G(N)$ for some matroid $N$.
\end{defi}

This property leads us to study the polytope $P_{M,M^\ast}$ in more depth in~\Cref{sec:poly}. Our first result shows that the property \textsc{Mat} is the only case in which a straight-forward polytopal analysis for Hamiltonian connectivity is feasible, and is related to the identically self-dual matroids studied in~\cite{geiger2023graph,Perrott2017,geiger2022selfdual,Lin84}. 
\begin{itemize}
    \item \textsc{Mat} holds for a block-matroid $M$ if and only if $G(M,M^\ast)$ is a subgraph of the $1$-skeleton of $P_{M,M^*}$ if and only if $M$ admits a special weak map to an identically self-dual matroid (\Cref{selfdual}).
\end{itemize}

Further polytopal considerations enable  us to show that if a minor-closed class satisfies some property of \Cref{conjbc}, then so does the class of its series-parallel extensions (\Cref{serparpreserve}). This allows us to give the first main result:

\begin{itemize}
    \item \textsc{Mat} holds for $\spex(\mathrm{LPM})$ (\Cref{thm:matroidal}), in particular \textsc{Ham} (and each property in \Cref{conjbc}) is satisfied in this class.
\end{itemize}

Next, we study a class of multipath matroids called whirls and a class of graphic matroids called wheels. We first observe:

\begin{itemize}
    \item \textsc{Mat} is not satisfied by wheels nor by whirls of rank at least $3$ (\Cref{lowerbound}).
\end{itemize}

We continue to study \textsc{Ham} for these classes:

\begin{itemize}
    
    \item \textsc{Ham} holds for wheels and whirls (\Cref{hamwheels} and \Cref{hamwhirls}). We do this by completely describing the structure of the base-cobase graphs of wheels and whirls (\Cref{prop:strucwheel} and \Cref{prop:strucwhirls}).
\end{itemize}

Towards regular matroids, we first observe that they rarely satisfy \textsc{Mat}:

\begin{itemize}
    \item \textsc{Mat} is satisfied by a regular matroid $M$ if and only $M$ is the direct sum of copies of $U_{1,2}$ (\Cref{regmat}).
\end{itemize}

Finally, we find a negative answer to~\cite[Problem (iii)]{Farber1985} of Farber, Richter, Shank within the class of regular matroids.
\begin{itemize}
    \item \textsc{Ham} fails for regular matroids (\Cref{prop:errediez}). We do this by explicitly providing the structure of $G_{R_{10},R_{10}^\ast}$ (\Cref{thm:r10desc}).
\end{itemize}
To our knowledge this is the first result that refutes a property from~\Cref{conjbc}.

\section{A polytopal glimpse}\label{sec:poly}

For the basic notions of polytopes we refer to~\cite{Zie95}.
Recall that the base polytope $P_M$ of a matroid $M=(E,\mathcal{B})$ is the convex hull of the incidence vectors $x_B\in\{0,1\}^E$ of all $B\in\mathcal{B}$. The following alternative description is well known.

\begin{prop}\label{flat-description}
A halfspace description of a matroid polytope $P_M\subseteq\mathbb{R}^n$ of $M=([n],\mathcal{B})$ is:
\begin{itemize}
    \item $\sum_{i\in [n]}x_i=r$,
    \item $x_i\geq 0$ for all $i\in[n]$,
    \item $\sum_{i\in F}x_i\leq r(F)$ for all flats $F$ of $M$.
\end{itemize}
\end{prop}

Feichtner and Sturmfels \cite{FS05} found a more precise description via the facet defining halfspaces that we now explain.  
Given a matroid $M=(E,\mathcal{B})$ with rank function $r$, a \emph{flat} is a subset $F\subseteq E$ such that $r(F\cup\{e\})>r(F)$ for all $e\in E\setminus F$. A flat $F$ is called \emph{separable} if there exists a non-trivial partition $F_1\cup F_2=F$ such that $r(F)=r(F_1)+r(F_2)$ and \emph{inseparable} otherwise. Finally, $F\subseteq E$ is called a \emph{flacet} if $F$ and $\overline{F}$ are inseparable flats of $M$ and $M^*$, respectively.

\begin{prop}[Feichtner and Sturmfels]\label{H-description}
The facet-defining halfspaces of a matroid polytope $P_M\subseteq\mathbb{R}^n$ of $M=([n],\mathcal{B})$ of rank $r$ are of the form:
\begin{itemize}
    \item $\sum_{i\in [n]}x_i=r$,
    \item $x_i\geq 0$ for all $i\in[n]$,
    \item $\sum_{i\in F}x_i\leq r(F)$ for all flacets $F$ of $M$.
\end{itemize}
\end{prop}

As corollary to Edmond's matroid intersection theorem \cite{Edm70}, for any two matroids $M_1,M_2$ it happens that the convex hull of incidence vectors of common bases of $M_1$ and $M_2$ is $P_{M_1} \cap P_{M_2}$, see \cite[Corollary 41.12d]{Sch03}. In particular, we get the following fact that was claimed without proof by Cordovil and Moreira~\cite{Cordovil1993}: 
    
\begin{prop}\label{prop:edmonds}
    For any matroid $M$ we have that $$P_{M,M^\ast}=P_M \cap P_{M^\ast}.$$ 
\end{prop}

B{\'e}rczi, M{\'a}trav{\"o}lgyi, and Schwarcz~\cite{brczi2023reconfiguration} also introduced the notion of \textit{tight sets} of a matroid as those $F\subseteq E$ such that $|F|=2r(F)$. The following result characterizes when a matroid has a non-trivial tight set in terms of the dimension of the base-cobase polytope.

\begin{lem}\label{codim}
For a connected block matroid $M=(E,\mathcal{B})$ the following are equivalent:
\begin{itemize}
    \item[(i)] There exists a non-trivial tight subset $F\subseteq E$.
    \item[(ii)] The base-cobase polytope satisfies $\dim(P_{M,M^*})<|E|-1$.
    \item[(iii)] There is a non-trivial flacet $F$ of $E$ such that $M_{F}$, $M^*/F$ are block matroids and for $M'=M_{F}\oplus M^*/F$ we have that $P_{M,M^*}=P_{M',M'^*}$.

\end{itemize}
\end{lem}
\begin{proof}~

\noindent (i)$\Rightarrow$(ii)] Let $\emptyset\neq F\subsetneq E$ be such that $2r(F)=|F|$. We begin by proving that $F$ is a flat of $M$. Let $B$ be a base-cobase of $M$. The restrictions of $B$ and $\overline{B}$ to $F$ are independent and their union is $F$, so $B\cap F$, $\overline{B}\cap F$ must be both of size $|F|/2$. For $e\in \overline{F}$, note that $F\cup \{e\}$ contains an extension of either $B\cap F$ or $\overline{B}\cap F$, so $r(F\cup \{e\})>r(F)$. This proves that $F$ is a flat. The previous arguments also prove that the incidence vectors $x_B$ and $x_{\overline{B}}$ satisfy with equality the inequality $$\sum_{i\in F} x_i \leq  \frac{|F|}{2} = r(F),$$ which forms part of a halfspace description of $P_M$ by \Cref{flat-description}. Then, the center $c=\frac{x_B+x_{\overline{B}}}{2}$ of the $\{0,1\}$-cube satisfies this equality as well. Since $P_{M^\ast}$ is the reflection of $P_M$ with respect to $c$, then all the points of $P_{M^\ast}$ lie in the halfspace $$\sum_{i\in F} x_i \geq  \frac{|F|}{2}.$$ But then all the points in $P_{M,M^\ast}=P_M\cap P_{M^\ast}$ (\Cref{prop:edmonds}) satisfy both inequalities, and thus they all satisfy the equality $\sum_{i\in F} x_i =  \frac{|F|}{2}$. Since they also satisfy the linearly independent equality $\sum_{i\in E} x_i= r$, we have that $\text{dim}(P_{M,M^\ast})\leq |E|-2$, as desired.

\noindent (ii)$\Rightarrow$(iii)] Denote $n=|E|$ and let $B$ be a base-cobase of $M$. Since $M$ is connected, we have $\dim(P_M)=n-1$. So for $\dim(P_{M,M^\ast})<n-1$ to happen, $P_{M,M^\ast}$ must be in a hyperplane $H$ linearly independent to $\sum_{i\in E} x_i=r$. We have $x_B$ and $x_{\overline{B}}$ in $H$. Also, $H$ must be a support hyperplane of $P_M$, as otherwise $P_{M,M^\ast}=P_M\cap P_{M^\ast}$ would not be contained in $H$. Therefore, $B$ and $\overline{B}$ lie on a common face of $P_M$, and hence in a common facet, which by~\Cref{H-description} corresponds to a flacet $F$ of $M$. Since for $x_B$ and $x_{\overline{B}}$ the equality $\sum_{i\in F}x_i=r(F)$ holds, it follows that $|B\cap F|=|\overline{B}\cap F|=\frac{1}{2}|F|$. A standard rank calculation shows that $\overline{F}$ has rank $\frac{|\overline{F}|}{2}$ in $M^\ast/F$ and then that $|B\cap \overline{F}|=|\overline{B}\cap \overline{F}|=\frac{|\overline{F}|}{2}$. Then $M_{F}\oplus M/F$ is a matroid whose set of base-cobases is $\mathcal{B}\cap \mathcal{B}^*$. 
        
\noindent (iii)$\Rightarrow$(i)] Let $F$ be a non-trivial flacet of $F$ as in (iii). In particular, $M_F$ is a block matroid, so $r(F)=\frac{|F|}{2}$, as desired.
\end{proof}

A matroid $N=(E,\mathcal{B})$ is called \emph{identically self-dual} if the identity map is an isomorphism between $N$ and $N^*$, i.e. $N=N^*$.

\begin{lem}[\textsc{Mat} is rare]\label{selfdual}
Let $M=(E,\mathcal{B}_M)$ be a block matroid, then the following are equivalent:
\begin{itemize}
    \item[(i)] $G(M,M^\ast)$ is the base graph of a matroid $N$, i.e., $M$ satisfies \textsc{Mat},
    \item[(ii)] $G(M,M^*)$ is the $1$-skeleton of $P_{M,M^*}$,
    \item[(iii)] there exists an identically self-dual matroid $N$, such that $\mathcal{B}_N\subseteq \mathcal{B}_M$ and $\mathcal{B}_N\cap\mathcal{B}_{N^*}=\mathcal{B}_M\cap\mathcal{B}_{M^*}$.
\end{itemize}
\end{lem}
\begin{proof}~

\noindent (i)$\Rightarrow$(ii)] We have $P_{M,M^*}$ and $P_N$ have the same vertex set. Thus, $G(M,M^*)=G(N)=G(P_N)=G(P_{M,M^*})$. Here we use that the base graph of a matroid is isomorphic to the $1$-skeleton of the base polytope.

\noindent (i)$\Leftarrow$(ii)] If $G(M,M^*)$ is the $1$-skeleton of $P_{M,M^*}$, then $P_{M,M^*}$ is a  $(0,1)$-polytope with all edge directions parallel to a difference $e_i-e_j$ for canonical vectors $e_i,e_j$. Hence, by~\cite{Edm70,GGMS87} $P_{M,M^*}=P_N$ for some matroid $N$.

\noindent (i)$\Leftrightarrow$(iii)] We have that $\mathcal{B}_N=\mathcal{B}_M\cap\mathcal{B}_{M^*}$ is equivalent to $\mathcal{B}_N=\mathcal{B}_N\cap\mathcal{B}_{N^*}$, which is equivalent to $N$ being identically self-dual. 
\end{proof}

\begin{quest}
Is there a nice characterization of the class of matroids satisfying \textsc{Mat}?
\end{quest}



\section{Series parallel extensions and Lattice Path Matroids}
\label{sec:spex}

We say that a matroid $M$ on $E$ is a \textit{parallel extension} of a matroid $N$ on $F$ if there exists an element $e\in F$ in a circuit of size $2$ of $M$ such that $M=N\setminus e$. The dual notion is being a \textit{series extension}. For this, there must be an element $e\in F$ such that $e$ is in a cocircuit of size $2$ of $M$ and $M=N/e$. 
For a class $\mathbf{M}$ of matroids, denote by $\spex(\mathbf{M})$ its closure under (finite) series-parallel extensions.
\begin{lem}\label{minors}
 If $\mathbf{M}$ is a class of matroids that is closed under minors and direct sums with loops and coloops, then we have that $\spex(\mathbf{M})$ is closed under minors. 
\end{lem}
\begin{proof}
Suppose that $N$ is a series-parallel extension of $M\in\mathbf{M}$ and let $N'<N$ be a minor. Denote by $S,P\subseteq N$ the elements that are series and parallel extensions of $M$ and let $C,D\subseteq N$ be the elements that are contracted and deleted, respectively, to obtain $N'$. Then it is straight-forward to check that $N'$ is a series-parallel extension of an element of $\mathbf{M}$ via $$N'/(P\setminus C)\setminus(S\setminus D)=M/(C\setminus P)\setminus(D\setminus S)\oplus |(C\cap P)|\cdot U_{0,1}\oplus |(D\cap S)|\cdot U_{1,1},$$ where the latter coefficients stand for the number of loop and co-loop summands. 
\end{proof}

Next, we need a lemma briefly mentioned in the graphic matroid case in \cite{Farber1985}. For completeness, we provide here the general matroid version and a proof.

\begin{lem}
\label{lem:factor}
    Let $M$ be a block matroid of rank $r\geq 1$ with groundset $E$. Suppose $M$ has a cocircuit (dually, a circuit) of size $2$, say $C=\{e,f\}$. Let $N=M\setminus C$ (dually, $N=M/C$). Then $N$ is a block matroid of rank $r-1$ and 
$$\mathcal{B}_M\cap\mathcal{B}_{M^*}=\mathcal{B}_{M'}\cap\mathcal{B}_{M'^*}$$ for $M'=U_{1,2}\oplus N$.

\end{lem}

\begin{proof}

We prove the version for circuits. The version for cocircuits will follow by duality. Simply note that $C$ has size $2$ and rank $1$, so $C$ is a tight set. The result then follows from~\Cref{codim}. 
\end{proof}

It is easy to see that block matroids are closed under direct sums, and that components of block matroids are block matroids themselves. This gives that if for a block matroid we have $M=M_1\oplus M_2$, then $$G(M,M^*)=G(M_1,M_1^*)\square G(M_2,M_2^*),$$ where $\square$ denotes the Cartesian product of graphs. The rank function of a matroid is additive with respect to $\oplus$. Also, the distance in a Cartesian product of graphs is the sum of distances along each coordinate. Finally, also the Cartesian product of Hamiltonian connected graphs or hypercubes is Hamiltonian connected or a hypercube. Hence, we have the following essential observation.
\begin{obs}\label{sums}
 Let $M_1,M_2$ be block matroids satisfying some property from~\Cref{conjbc} or \textsc{Mat}. Then so does their direct sum $M_1\oplus M_2$.
\end{obs}

\begin{lem}\label{serparpreserve}
 Let $\mathbf{M}$ be a class of matroids that is closed under minors and direct sums with loops and coloops. If $\mathbf{M}$ satisfies a property from~\Cref{conjbc} or \textsc{Mat}, then so does $\spex(\mathbf{M})$.
\end{lem}
\begin{proof}
We proceed by induction on the size $n$ of the groundset $E$ of $M$. If $n=0$, the result is true. Suppose the result is true for any block matroid in $\spex(\mathbf{M})$ whose ground set has size less than $n$, and let $M$ be a block matroid in $\spex(\mathbf{M})\setminus\mathbf{M}$ with ground set $|E|=n$. If $M\in \mathbf{M}$, we are done. Thus $M$ is a series or parallel extension of a smaller matroid in ${\spex(\textbf{M}})$. By duality we only consider the case that $M$ has a circuit of size $2$, say $C=\{e,f\}$.
Using~\Cref{lem:factor}, we get that 
$$\mathcal{B}_M\cap\mathcal{B}_{M^*}=\mathcal{B}_{M'}\cap\mathcal{B}_{M'^*}$$ for $M'=U_{1,2}\oplus N$.
Note that $N$ is in $\spex({\textbf{M}})$ because this class is closed under minors by~\Cref{minors}. By the  inductive hypothesis, we know that $N$ satisfies the desired property. Further, $U_{1,2}$ satisfies all properties from~\Cref{conjbc} and \textsc{Mat}. The rest follows from~\Cref{sums}.
\end{proof}

A \emph{lattice path matroid} is given by a pair of non-crossing monotone lattice paths $U$ and $L$ from $(0,0)$ to $(m,r)$. This is typically depicted as a diagram in the plane grid, which is bounded above by $U$ and below by $L$. Any monotone lattice path from the bottom left to the upper right corner inside this diagram is identified with a set $B$, where $i\in B$ if and only if the $i$-th step of the path is north. Now, the collection $\mathcal{B}$ of these sets forms the set of bases of a matroid denoted as $M[U,L]$. 

\begin{prop}\label{thm:LPM}
For any $M=M[U,L]$ a block lattice path matroid, there is a block lattice path matroid $N$ such that $G(M,M^*)=G(N)$. 
\end{prop}
\begin{proof}
Let $\overline{U}$ and $\overline{L}$ be the reflections $U$ and $L$ with respect to the line $y=x$. Let $U'=\min(U,\overline{L})$ be the lower envelope of $U$ and $\overline{L}$ and $L'=\max(\overline{U},L)$ be the upper envelope of $\overline{U}$ and $L$.

We claim that a lattice path $P$ is a base-cobase of $M$ if and only if $P$ lies in between $U'$ and $L'$. Indeed, in order for $P$ to be a base of $M$, it must lie above $L$ and below $U$. In order for $P$ to be a cobase, it must be a base of $M^\ast$. It is known~\cite{Bon-03} that $M^*$ corresponds to the lattice path matroid $M[\overline{L},\overline{U}]$. Thus, $P$ must also be above $\overline{U}$ and below $\overline{L}$. Therefore, a sufficient and necessary condition for $P$ to be a base-cobase is to lie in between $U'$ and $L'$. This shows that the base-cobases of $M$ are precisely the bases of $N:=M[U',L']$.

\end{proof}

\begin{thm}[\textsc{Mat} holds in $\spex(\mathrm{LPM})$]
\label{thm:matroidal} Let  $M\in\spex(\mathrm{LPM})$ be a block matroid. Then $G(M,M^\ast)$ is the base graph of a matroid in $\spex(\mathrm{LPM})$. 
\end{thm}

\begin{proof}
It is well-known that $\mathrm{LPM}$ is closed under minors and direct products, see e.g.~\cite{Bon-03}. Further, a loop and a coloop are in $\mathrm{LPM}$. Thus, by~\Cref{minors} $\spex(\mathrm{LPM})$ is minor-closed. The rest follows combining~\Cref{thm:LPM} and ~\Cref{serparpreserve}.
\end{proof}

\section{Wheels and whirls}\label{WW}

\subsection{Definitions}

Considering \Cref{fig:ww} we see two rows that depict three families of matroids. The first row shows the well-known family of graphic matroids of \emph{wheels}. For our purposes, for $n\geq 3$ denote by  $M(\mathcal{W}_n)$ the graphic matroid of the wheel graph $\mathcal{W}_n$ with vertex set $c, v_1,\ldots,v_n$ with an edge $2i-1$ between $v_i$ and $v_{i+1}$ for each $i\in[n]$ (indices modulo $n$), and an edge $2i$ between vertices $c, v_i$ for each $i\in [n]$. We call each edge $e_i=v_iv_{i+1}$ ($i=1,\ldots,n$, indexes modulo $n$) a \emph{rim} edge and each edge $f_i=cv_i$ ($i=1,\ldots,n$) a \emph{spoke}. We may abuse the name and call $M(\mathcal{W}_n)$ a wheel as well. It follows from a result of Bondy~\cite{Bon72} that wheels are not transversal matroids. The second family is also depicted by the first row of the figure. It is the family of \textit{whirls}. The whirl $\mathcal{W}^n$ is the matroid obtained from $M(\mathcal{W}_n)$ by declaring the outer cycle a basis. It is well-known (see \cite{Oxley2006}) that this yields a matroid.

The second row depicts our third  family of matroids, a particular subfamily of transversal matroids. For a positive integer $n\geq 3$, consider the integers $\{1,2,\ldots,2n\}$ and the cyclic intervals $I_1=\{1,2,3\}$, $I_2=\{3,4,5\}$, $\ldots$, $I_{n}=\{2n-1,2n,1\}$. The \textit{necklace $N_n$} is the transversal matroid presented by $I_1,\ldots,I_n$. Each $N_n$ is a block matroid. The figure depicts these matroids for $n=4,6,8,10,12$. Since the considered sets are circular intervals, the matroids we obtain are multipath matroids and in particular positroids, see~\cite{B07,LP07,Bon-07}. Moreover, since each element is contained in at most two sets they are also bicircular, see~\cite{Mat77}. However, there are no graphic or cographic matroids in this family since $N_n|_{\{1,2,3,4\}}$ is a $U_{2,4}$ minor.

\begin{figure}[ht]
    \centering
    \includegraphics[width=\textwidth]{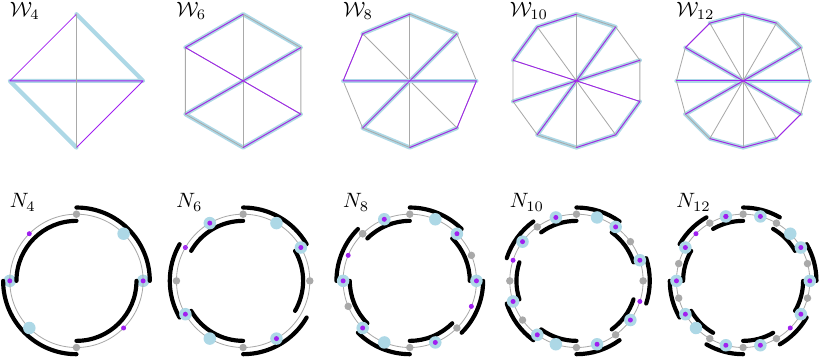}
    \caption{The first row depicts wheels and whirls. The second row depicts the necklaces $N_n$. In each instance, we highlight two base-cobases that show $d_{G(M,M^*)}({B},{B'})=(\frac{r}{4}+\frac{1}{2})d_{G(M)}({B},{B'})$.}
    \label{fig:ww}
\end{figure}

It was shown by Bonin and Giménez \cite{Bon-07} that $N_6$ and $\mathcal{W}^3$ are isomorphic. It is a folklore result that this is true in general, for $N_n$ and $\mathcal{W}^n$. In~\Cref{sec:whirlstransv} we provide a proof for the sake of completeness.

\subsection{Structure of the base-cobase graph of wheels and whirls}

In this section, we describe the structure of the base-cobase graph for wheels and whirls.  In~\cite[Section 6]{AHM14}, the authors describe and count the base-cobases of wheels and whirls. Also, part of the structure of the base-cobase graph of wheels is used in~\cite{BMS24} to bound weighted exchange distances. Another paper studying their structure is~\cite{AHM14}. We will provide a full description in terms of hypercube graphs that will be useful for proving Hamiltonian connectivity later on.

Let $n$ be a positive integer and $B_n$ the set of all binary vectors of length $n$. The \emph{$n$-dimensional hypercube graph $\mathcal{Q}_n$} is the graph with $B_n$ as vertex set where two vertices are adjacent if their corresponding vectors differ in exactly one component. If this component is $i\in[n]$, then we call the corresponding edge an \emph{$i$-edge}. For a binary vector $v$, its \emph{support} $\sup(v)$ is the set of components of $v$ equal to one (e.g. $\sup((1,1,0,1,0))=\{1,2,4\}$).

We define $C_n$ as the set of binary vectors $v$ of length $n$ such that $\sup(v)$ is a cyclic interval of $\{1,2,\ldots,n\}$. We call the graph induced by $C_n$ the \emph{lean subgraph of $\mathcal{Q}_n$}. We also refer to edges and vertices within $C_n$ as \emph{lean edges} and \emph{lean vertices} of $\mathcal{Q}_n$. The following easy observation shows that there are many lean edges using a given coordinate, which in the future will allow us to pick such an edge rather freely.

\begin{obs}\label{obs:manyleanedges}
    Let $n$ be a positive integer and $1\leq i\leq n$ a coordinate. There are $2(n-2)$ lean $i$-edges in $\mathcal{Q}_n$.
\end{obs}

Further, we call $\mathbf{1}$ and $\mathbf{0}$ the vectors of all $1$'s and all $0$'s, respectively. For $0\leq j \leq n$,  we define $B_{n,j}$ (resp. $C_{n,j}$) as the vectors in $B_n$ (resp. $C_n$)  having exactly $j$ ones.

In what follows, we will need two disjoint copies of $\mathcal{Q}_n$, which we will call $\mathcal{Q}_n^+$ and $\mathcal{Q}_n^-$. Thus, we will use $B_n^+$, $C_n^+$, etc. to refer to objects from the first copy and $B_n^-, C_n^-$, etc. to refer to objects from the second copy.

The result below states that the base-cobase of a wheel are two almost-hypercubes \emph{stitched} together at their lean subgraphs.

\begin{prop}
\label{prop:strucwheel}
Let $n$ be a positive integer. The base-cobase graph of the wheel $M(\mathcal{W}_n)$ is obtained from the disjoint union of the graphs $\mathcal{P}_n^+:=\mathcal{Q}_n^+-\{\mathbf{0}^+, \mathbf{1}^+\}$ and $\mathcal{P}_n^-:=\mathcal{Q}_n^--\{\mathbf{0}^-, \mathbf{1}^-\}$, and the following \emph{stitching edges}. For each $v \in C_n - \{\mathbf{0}, \mathbf{1}\}$ with $\sup(v)=\{k+1, k+2, \ldots, k+\ell\}$ (we take the positions modulo $n$), we add an edge from $v^+$ to each $w^-$ such that $w \in C_n - \{\mathbf{0}, \mathbf{1}\}$ and:

\begin{itemize}
    \item $\sup(w)=\{k+1, \ldots, k+\ell\}$ (i.e. to $v^-$).
    \item $\sup(w)=\{k+2, \ldots, k+\ell+1\}$.
    \item $\sup(w)=\{k+2, \ldots, k+\ell\}$, if $\ell\geq 2$.
    \item $\sup(w)=\{k+1, \ldots, k+\ell+1\}$, if $\ell\leq n-2$.
\end{itemize}

\end{prop}

\begin{proof}
    Consider the wheel graph $\mathcal{W}_n$ as described above. Each base-cobase graph is in bijective correspondence with a vertex in $\mathcal{P}_n^+\cup \mathcal{P}_n^-$ as follows:
    
    \begin{itemize}
        \item If $v^+\in \mathcal{P}_n^+$, then include the rim edges $e_i$ for $i\in \sup(v)$,  the spokes $f_j$ for each $j\not\in \sup(v)$ and the spokes $f_j$ for each position $j$ that starts a clockwise cyclic interval of $1$'s of $v^+$.
        \item If $v^-\in \mathcal{P}_n^-$, then include the rim edges $e_i$ for $i\in \sup(v)$,  the spokes $f_j$ for each $j\not\in \sup(v)$ and the spokes $f_{j+1}$ for each position $j$ that ends a clockwise cyclic interval of $1$'s of $v^-$.
    \end{itemize}

    If two vertices $u^+$ and $v^+$ in $\mathcal{P}_n^+$ differ in exactly one coordinate, the corresponding base-cobases differ in exactly one rim and one spoke, so one can be obtained from the other by a symmetric exchange. The same happens for vertices in $\mathcal{P}_n^-$. See \Cref{fig:heptagons2} for local examples in the case $n=7$.
    
    \end{proof}


\begin{figure}
    \centering
    \includegraphics[width=.8\linewidth]{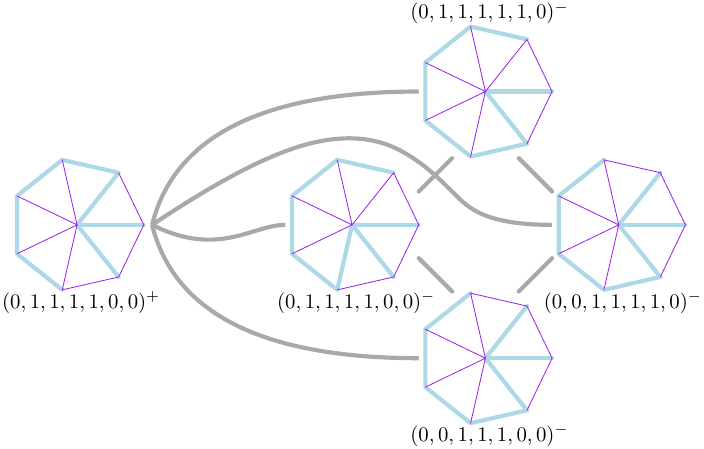}
    \caption{A positive base-cobase of the wheel with its neighboring negative base-cobases, and the corresponding subgraph in $\mathcal{P}_7^+\cup \mathcal{P}_7^-$.}
    \label{fig:heptagons2}
\end{figure}

The base-cobase graphs of whirls are nearly identical. 

\begin{prop}\label{prop:strucwhirls}
Let $n$ be a positive integer. The base-cobase graph of the whirl $\mathcal{W}^n$ is obtained starting from the disjoint union of the graphs $\mathcal{Q}_n^+$ and $\mathcal{Q}_n^-$, then identifying $\mathbf{0}^+$ with $\mathbf{0}^-$ into a single vertex $\mathbf{0}$ and $\mathbf{1}^+$ with $\mathbf{1}^-$ into a single vertex $\mathbf{1}$, and finally adding the same stitching edges as for wheels.
\end{prop}

For a matroid $M$, in its base graph $G(M)$, by the exchange axiom the distance between bases $B$ and $B'$ is exactly $\frac{|B\triangle B'|}{2}$. As already noted in~\cite{BMS24}, for base-cobase graphs of wheels the distance of some base-cobases can be larger. The description above gives another point of view on why this happens, and extends the results to whirls as well. Even though the next result follows almost directly from~\cite{BMS24}, we present it as a warm-up for our notation.

\begin{prop}[\textsc{Mat} does not hold for wheels or whirls]\label{lowerbound}
Let $n\geq 3$ and $M\in \{M(\mathcal{W}_n),\mathcal{W}^n\}$. Then there are base-cobases ${B^+},{B^-}$ with $\frac{|B^+\triangle B^-|}{2}=2$, but $$d_{G(M,M^*)}(B^+,B^-)=2\left\lfloor\frac{n}{4}\right\rfloor+1.$$
\end{prop}
\begin{proof}
    Express $n=4\ell+k$ for $k\in\{0,1,2, 3\}$. We give the proof in the case $k=0$, as the other cases require only minor adaptations. Consider the vertex $v=(a,b,c,d)$ of $\mathcal{Q}_{4\ell}$ where $a$ and $c$ are blocks of $\ell$ entries equal to $1$ and $b$ and $d$ are blocks of $\ell$ entries equal to $0$. Let $v^+$ and $v^-$ be the corresponding vertices in $\mathcal{Q}_{4\ell}^+$ and $\mathcal{Q}_{4\ell}^-$. The corresponding base-cobases $B^+$ and $B^-$ of the matroid differ only at $4$ spokes, so $\frac{|B^+\triangle B^-|}{2}=2$. However, to get from $v^+$ to $v^-$ we must go from $v^+$ to a lean vertex $w^+$ in $\mathcal{Q}_{4\ell}^+$, which needs either adding $\ell$ ones, or removing $\ell$ ones, then moving to a lean vertex $w^-$ in $\mathcal{Q}_{4\ell}^-$, and finally moving to $v^-$, which again needs at least $\ell$ movements. In total, we need at least $2\ell+1$ movements to go from $v^+$ to $v^-$, and clearly $2\ell+1$ are enough.
\end{proof}



\subsection{Hamiltonian connectivity of the base-cobase graphs of wheels and whirls}
Recall that a graph is Hamiltonian connected if for any two vertices there is a Hamiltonian path starting at one and ending at the other. If a graph on at least 3 vertices is bipartite, then it cannot be Hamiltonian connected. 
Fortunately, the base-cobase graphs of wheels and whirls are not bipartite. Indeed, by 
\Cref{prop:strucwheel} a typical vertex of $C_n^+\cup C_n^+$ is contained in a triangle.

Naddef and Pulleyblank proved a very general result on Hamiltonian connectivity, stating that the $1$-skeleton of a $(0,1)$-polytope is either Hamiltonian connected or a $\mathcal{Q}_n$ \cite{NO84}. This implies that a matroid is either the direct sum of $U_{1,2}$'s, or its base graph is Hamiltonian connected. However, \Cref{lowerbound} shows that the collection of base-cobases of wheels (or whirls) is not the base set of a matroid, so in view of \Cref{selfdual}, we have that $G(M,M^\ast)$ is not the $1$-skeleton of $P_{M,M^\ast}$.

Since by~\Cref{prop:strucwheel} the base-cobase graphs of wheels and whirls is made up by hypercubes, we will need their Hamiltonian properties. Even though $\mathcal{Q}_n$ is not Hamiltonian connected, it is robust in terms of properties related to Hamiltoninan connectivity, in the sense that even after removing edges or vertices, a suitable relaxation of Hamiltoninan connectivity still holds. 

We will need some very specific of these properties all due to Castañeda and Gotchev~\cite{castaneda2015path}. For the statement we denote the  bipartition of the hypercube $\mathcal{Q}_n=R\cup G$ into colors \emph{red} and \emph{green}. It is completely determined by setting  $\mathbf{0}$ as a green vertex. The vertices in the layer $B_{n,j}$ are green (in $G$) for $j$ even, and red (in $R$) otherwise. We further need the notion of a \emph{$k$-path covering}, i. e., a family of $k$ paths covering each vertex exactly once, and having some prescribed endpoints.
Note that the following statements also hold interchanging red and green.

\begin{thm} Let $n\geq 4$ be a positive integer and $\mathcal{Q}_n$ the $n$-dimensional hypercube graph. Let $R$ and $G$ be the bipartition as described above. Then:
\label{thm:auxiliares}
    \begin{enumerate}[label=(\alph*), ref=\thethm~(\alph*)]
    
        \item \label{item:3.2}\cite[Corollary 3.2]{castaneda2015path} For any $r$ in $R$, $g$ in $G$ and an edge $e$ different from $\{r, g\}$, there exists a Hamiltonian path of $\mathcal{Q}_n$ from $r$ to $g$, and through $e$.
        \item \label{item:3.6}\cite[Lemma 3.6]{castaneda2015path} For any distinct $r,r^\ast$ in $R$, $g,g^\ast$ in $G$, there exists a Hamiltonian path in $\mathcal{Q}_n-\{r^\ast,g^\ast\}$ from $r$ to $g$. 
        \item \label{item:3.11}\cite[Lemma 3.11]{castaneda2015path} Let $r^\ast\in R$. For any $r_1$ in $R-\{r^\ast\}$ and three distinct $g_1$, $g_2$, $g_3$ in $G$, there exists a $2$-path covering of $\mathcal{Q}_n-\{r^\ast\}$ with one path connecting $r_1$ to $g_1$ and the other connecting $g_2$ to $g_3$. 
        \item \label{item:3.12}\cite[Lemma 3.12]{castaneda2015path} Let $g^\ast \in G$. For any distinct $r_1$, $r_2$ in $R$ and any edge $e$ in $\mathcal{Q}_n-\{g^\ast\}$, there exists a Hamiltonian path in $\mathcal{Q}_n-\{g^\ast\}$ from $r_1$ to $r_2$, and through $e$.
        
        \item \label{item:3.13}\cite[Lemma 3.13]{castaneda2015path} Let $g^\ast$ in $G$ and $r_1^\ast, r_2^\ast$ distinct elements in $R$. For any distinct $g_1$, $g_2$ in $G - \{g^\ast\}$, there exists a Hamiltonian path in $\mathcal{Q}_n - \{g^\ast, r_1^\ast, r_2^\ast\}$ from $g_1$ to $g_2$.
        \item \label{item:4.3}\cite[Lemma 4.3]{castaneda2015path} Let $r_1^\ast,r_2^\ast$ be two distinct elements in $R$. For any distinct $g_1$, $g_2$, $g_3$, $g_4$ in $G$, there exists a $2$-path covering of $\mathcal{Q}_n-\{r_1^\ast, r_2^\ast\}$ with one path from $g_1$ to $g_2$ and the other from $g_3$ to $g_4$. 
        \item \label{item:5.6}\cite[Lemma 5.6]{castaneda2015path} Let $r^\ast\in R$ and $g^\ast \in G$. For any distinct $r_1$, $r_2$, $g_1$, $g_2$ in $G$, there exists a $2$-path covering of $\mathcal{Q}_n-\{r^\ast, g^\ast\}$ with one path from $r_1$ to $g_1$ and the other from $r_2$ to $g_2$. 
        \item \label{item:5.10}\cite[Corollary 5.10]{castaneda2015path} 
        Suppose $n\geq 6$. Let $r_1^\ast,r_2^\ast$ be two distinct elements in $R$. For any distinct $g_1, g_2, g_3, g_4, g_5$ in $G$ and $r_1$ in $ R- \{r_1^\ast, r_2^\ast\}$, there exists a $3$-path covering of $\mathcal{Q}_n-\{r_1^\ast, r_2^\ast\}$ with one path from $g_1$ to $g_2$, one path from $g_3$ to $g_4$ and one path from $g_5$ to $r_1$.
    \end{enumerate}
\end{thm}

We give an extension of~\Cref{item:3.6} adapted to our purposes.

\begin{lem}\label{lem:lemma36plus}
     Let $n\geq 4$ be a positive integer and $\mathcal{Q}_n$ the $n$-dimensional hypercube graph. Let $R$ and $G$ be the bipartition as described above. Let $r^\ast$ in $R$ and $g^\ast$ in $G$. For any $r$ in $R-\{r^\ast\}$, any $g$ in $G-\{g^\ast\}$, there exists a Hamiltonian path in $\mathcal{Q}_n-\{r^\ast, g^\ast\}$ from $r$ to $g$, and through an edge $e$ of $C_n$ (i.e. a lean edge).   
\end{lem}

\begin{proof}
By detecting a coordinate in which $r$ and $r^\ast$ differ, we can find complementary $(n-1)$-dimensional hypercubes $Q$ and $Q^\ast$ as subgraphs of $\mathcal{Q}_n$ with $r$ in $Q$ and $r^\ast$ in $Q^\ast$. Then, we have the following cases for the locations of $g$ and $g^\ast$:

\begin{itemize}
    \item[a)] $g$ and $g^\ast$ in $Q^\ast$.
    \item[b)] $g$ in $Q^\ast$ and $g^\ast$ in $Q$.
    \item[c)] $g$ in $Q$ and $g^\ast$ in $Q^\ast$.
    \item[d)] $g$ and $g^\ast$ in $Q$.
\end{itemize}

Consider~\Cref{fig:lemma36plus} for an illustration of the cases and their resolution.

\begin{figure}[ht]
    \centering
    \includegraphics[width=.8\linewidth]{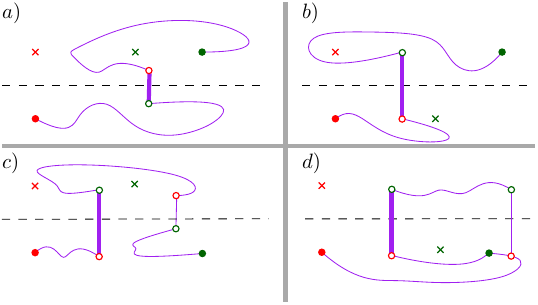}
    \caption{Cases in the proof \Cref{lem:lemma36plus}, $g,r$ are filled, $g^*,r^*$ are crosses, unfilled vertices are auxiliary, the thick purple edge is lean.}
    \label{fig:lemma36plus}
\end{figure}

We argue case by case. Having~\Cref{obs:manyleanedges} in mind, throughout the proof we will rather freely choose lean edges of the coordinate in which $r$ and $r^\ast$ differ. 

\paragraph{Case $a)$:} Pick any $g'\in Q$ with a lean edge $e$ towards a $r'\in Q^*\setminus\{r^\ast\}$. By~\Cref{item:3.2} we take a Hamilton path $P_1$ in $Q$ from $r$ to $g'$. Then by~\Cref{item:3.6} we can take a Hamilton path $P_2$ in $Q^*\setminus\{r^\ast,g^\ast\}$ from $r'$ to $g$. The concatenation $P_1,e,P_2$ satisfies the claim.

\paragraph{Case $b)$:} Pick any $r'\in Q\setminus\{r\}$ with a lean edge $e$ towards a $g'\in Q^*\setminus\{g\}$. By~\Cref{item:3.12} we take a Hamilton path $P_1$ in $Q\setminus\{g^\ast\}$ from $r$ to $r'$. Similarly, by~\Cref{item:3.12} we take a Hamilton path $P_2$ in $Q^\ast\setminus\{r\ast\}$ from $g'$ to $g$. The concatenation $P_1,e,P_2$ satisfies the claim.

\paragraph{Case $c)$:} Pick any $r'\in Q\setminus\{r\}$ with a lean edge $e$ towards a $g'\in Q^*\setminus\{g\}$, and furthermore $g''\in Q\setminus\{g\}$ with an edge $e'$ towards some $r''\in Q^\ast\setminus\{r^\ast\}$. By~\Cref{item:3.2} we take a 2-path covering $P_1, P_3$ of $Q$ with $P_1$ connecting $r$ to $r'$ and $P_3$ connecting $g''$ to $g$.
By~\Cref{item:3.6} we can take a Hamilton path $P_2$ in $Q^*\setminus\{r^\ast,g^\ast\}$ from $g'$ to $r''$. The concatenation $P_1,e,P_2,e',P_3$ satisfies the claim.

\paragraph{Case $d)$:} Pick any $r'\in Q\setminus\{r\}$ with a lean edge $e$ towards a $g'\in Q^*\setminus\{g^\ast\}$. By~\Cref{item:3.12} we take a Hamilton path $P_1$ in $Q\setminus\{g^\ast\}$ from $r$ to $r'$. Denote by $r''$ the predecessor of $g$ in $P_1$, by $g''$ the neighbor of $r''$ in $Q^\ast$, and $e=g''r''$.
Denote by $P'_1$ the subpath of $P$ from $r$ to $r''$ and by $P'_1$ the subpath of $P$ from $r'$ to $g$.
By~\Cref{item:3.12} we take a Hamilton path $P_2$ in $Q^\ast\setminus\{r^\ast\}$ from $g''$ to $g'$. The concatenation $P'_1,e',P_2,e,P''_1$ satisfies the claim.

\end{proof}

We will also need the following $2$-path covering result, with paths through specific edges.

\begin{lem}\label{lem:43plus}
 Let $n\geq 7$ be a positive integer and $\mathcal{Q}_n$ the $n$-dimensional hypercube graph. Let $R$ and $G$ be the bipartition as described above. Let $u_1,v_1,w_1$ be distinct vertices in $C_{n,1}$ and $u_2,v_2$ be distinct vertices in $C_{n,n-1}$. Let $z_1$ be any of the two neighbors of $w_1$ in $C_{n,2}$. 

 Then, there is a $2$-path covering of $\mathcal{Q}_n-\{\mathbf{0},\mathbf{1}\}$ with one path from $u_1$ to $v_1$ and through the edge $w_1z_1$, and the other path from $u_2$ to $v_2$.
\end{lem}

\begin{proof}
    There are at most $4$ coordinates in which some of the vectors $u_1,v_1,w_1,z_1$ have a $1$ in that position. Similarly, there are at most $2$ coordinates in which some of $u_2,v_2$ have a $0$ in that position. Since $n\geq 7$, there is at least one coordinate $k$ for which none of this happens.

    We split the cube $\mathcal{Q}_n$ into two $n-1$ dimensional hypercubes $Q, Q^\ast$ according to whether the $k$-th coordinate of a vector is either $0$ or $1$, respectively. Then, $\mathbf{0},u_1,v_1,w_1,z_1 \in Q$ and $\mathbf{1},u_2,v_2 \in Q^\ast$.

    By \Cref{item:3.12}, there is Hamiltonian path in $Q\setminus\{\mathbf{0}\}$ from $u_1$ to $v_1$ and through $w_1z_1$. Similarly, we can find a Hamiltonian path in $Q^\ast\setminus\{\mathbf{1}\}$ from $u_2$ to $v_2$.  The union of these two Hamiltoninan paths is our desired $2$-path covering of $\mathcal{Q}_n-\{\mathbf{0},\mathbf{1}\}$.
\end{proof}

When a path uses an edge $e$ between two vertices of $C_n^+$ (resp. two vertices of $C_n^-$), we can replace it by \emph{grabbing a vertex} from the opposite side, which means replacing it by two stitching edges and a vertex  of $C_n^-$ (resp. of $C_n^+$). We will use this trick in the following constructions.

\begin{figure}
    \centering
    \includegraphics[width=\linewidth]{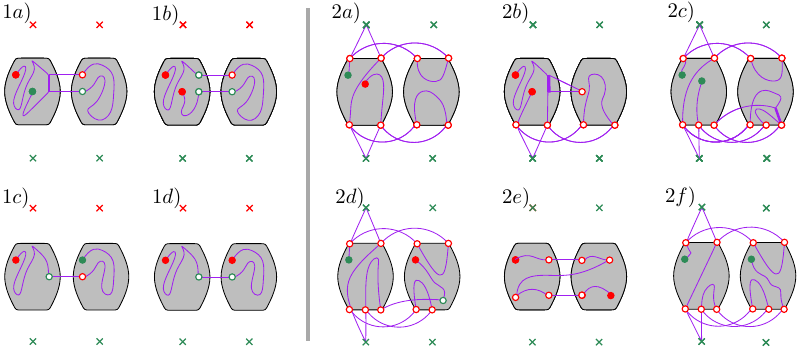}
    \caption{Cases in the proof of~\Cref{hamwheels}. Left: Odd wheels. Right: Even wheels.}
    \label{fig:wheels}
\end{figure}

We can now focus on proving that the base-cobase graphs of wheels are Hamiltonian connected.

\begin{thm}[Wheels satisfy \textsc{Ham}]\label{hamwheels}
    Let $n\geq 3$ be a positive integer. Then $$H:=G(M(\mathcal{W}_n), M(\mathcal{W}_n)^\ast)$$ is Hamiltonian connected.
\end{thm}

\begin{proof}
    The cases for $n\leq 6$ can be dealt with computationally, so we may assume that $n\geq 7$ and thus that we can use any of the auxiliary results above. Let $u$ and $v$ be distinct vertices of $H$.  Refer to~\Cref{fig:wheels} for illustration.
    
    \textbf{Case $1$:} $n$ odd. Considering the symmetry between $\mathcal{P}_n^+$ and $\mathcal{P}_n^-$, and the symmetry between colors within each graph, we have the following subcases to consider:

    \begin{itemize}
        \item \textbf{Case $1a)$:} $u,v$ lie in $\mathcal{P}_n^+$ and within this part, they have opposite colors, say $u\in R^+$ and $v\in G^+$. Here, using \Cref{lem:lemma36plus}, we find in $\mathcal{P}_n^+$ a Hamiltonian path $P_1$ from $u$ to $v$, and through a lean edge $e^+=\{r^+,g^+\}$. Let $e^-=\{r^-,g^-\}$ be the corresponding edge in $\mathcal{P}^-$. Using \Cref{item:3.2}, we find in $\mathcal{P}_n^-$ a Hamiltonian path $P_2$ from $r^-$ to $g^-$. The desired Hamiltonian path is obtained from replacing the edge $e^+$ of $P_1$ with the stitching edges $r^+r^-$, $g^+g^-$ and $P_2$. 
        \item \textbf{Case $1b)$:} $u,v$ lie in $\mathcal{P}_n^+$ and within this part, they have the same color, say $u,v \in R^+$. We choose distinct lean vertices $g_1^+$ and $g_2^+$ in $G^+$. Let $g_1^-$ be the corresponding vertex in $\mathcal{P}_n^-$ and $r_2^-$ be a vertex of $\mathcal{P}_n^-$ in $R^-$ adjacent to $g_2^+$. By \Cref{item:5.6}, there is a $2$-path covering of $\mathcal{P}_n^+$ with one path $P_1$ from $u$ to $g_1$ and the other path $P_2$ from $g_2$ to $v$. By \Cref{item:3.6}, we can find a Hamiltonian path $P_3$ in $\mathcal{P}_n^-$ from $g_1^-$ to $r_2^-$. The desired Hamiltonian path is the concatenation of $P_1$, $g_1^+g_1^-$, $P_3$, $r_2^-g_2^+$ and $P_2$.
        \item \textbf{Case $1c)$:} $u$ lies in $R^+$ and $v$ in $G^-$. We choose a lean green vertex $g^+$ and a red neighbor $r^-$ in $R^-$. We use \Cref{item:3.6} twice, to find a Hamiltonian path $P_1$ in $\mathcal{P}^+$ from $u$ to $g^+$, and to find a Haimiltonian path $P_2$ in $\mathcal{P}^-$ from  $r^-$ to $v$. The desired Hamiltonian path is the concatenation of $P_1$, $g^+r^-$ and $P_2$.
        \item \textbf{Case $1d)$:} $u$ lies in $R^+$ and $v$ in $R^-$. This case is analogous to the previous case, but using the copy $g^-$ instead of $r^-$.
    \end{itemize}

    \textbf{Case $2$:} $n$ even. Now we do not have symmetry between colors within each of $\mathcal{Q}_n^+$ and $\mathcal{Q}_n^-$, since both $\mathbf{0}$ and $\mathbf{1}$ are green. So we have a few extra subcases to study.

    \begin{itemize}
        \item \textbf{Case $2a)$:} $u,v$ lie in $\mathcal{P}_n^+$, $u$ is green and $v$ is red. By~\Cref{item:3.2}, we can find a Hamiltonian path $P_1$ in $\mathcal{Q}_n^+$ from $u$ to $v$. Let $r_1^+, r_2^+$ be the neighbors of $\mathbf{0}^+$ along $P_1$ (with $r_1^+$ before $r_2^+$), and $r_3^+,r_4^+$ be the neighbors of $\mathbf{0}^-$ along $P_1$ (with $r_3^+$ before $r_4^+$). Let $r_1^-,r_2^-,r_3^-,r_4^-$ be the corresponding copies in $\mathcal{P}_n^-$. By \Cref{item:4.3}, there is a $2$-path covering of $\mathcal{P}_n^-$ with one path $P_2$ from $r_1^-$ to $r_2^-$ and the other path $P_3$ from $r_3^-$ to $r_4^-$. The desired Hamiltonian path is obtained from $P_1$ by replacing $r_1^+\mathbf{0}^+r_2^+$ with the concatenation of $r_1^+r_1^-$, $P_2$ and $r_2^-r_2^+$, and replacing $r_3^+\mathbf{1}^+r_4^+$ with the concatenation of $r_3^+r_3^-$, $P_3$ and $r_4^-r_4^+$.
        
        \item \textbf{Case $2b)$:} $u,v$ lie in $\mathcal{P}_n^+$ and both are red. We choose a lean edge $e$ of $\mathcal{P}_n^+$ that does not use vertices of $C_{n,1}^+$ or $C_{n,n-1}^+$, and with which we can grab a red vertex $r^-$ from $\mathcal{P}_n^-$. Using \Cref{item:3.12}, we can find a Hamiltonian path $P_1$ in $\mathcal{Q}_n^+ - \{\mathbf{1}^+\}$ through edge $e$. Let $r_1^+$ and $r_2^+$ be the two red neighbors of $\mathbf{0}^+$ along $P_1$ (with $r_1^+$ before $r_2^+$). Let $r_1^-$ and $r_2^-$ be the corresponding vertices, respectively. Using \Cref{item:3.13} we can find a Hamiltonian path $P_2$ in $\mathcal{P}_n-\{r^-\}$ from $r_1^-$ to $r_2^-$. The desired Hamiltonian path is then obtained from replacing $e$ by grabbing $r^-$ in $P_1$, and replacing $r_1^+\mathbf{0}^+r_2^+$ with the concatenation of $r_1^+r_1^-$, $P_2$ and $r_2^-r_2^+$.

        \item \textbf{Case $2c)$:} $u,v$ lie in $\mathcal{P}_n^+$ and both are green. We choose $r^+$ in $C_{n,1}^+$. By \Cref{item:3.12} there is a Hamiltonian path $P_1$ in $\mathcal{Q}_n^+-\{r^+\}$ from $u$ to $v$. Let $r_1^+, r_2^+$ be the neighbors of $\mathbf{0}^+$ along $P_1$ (with $r_1^+$ before $r_2^+$), and $r_3^+,r_4^+$ be the neighbors of $\mathbf{1}^-$ along $P_1$ (with $r_3^+$ before $r_4^+$). Let $r^-,r_1^-,r_2^-,r_3^-,r_4^-$ be the corresponding copies in $\mathcal{P}_n^-$. Let $g^-$ be the vertex that allows the edge $r^-g^-$ to grab the vertex $r^+$.
        By \Cref{lem:43plus}, we can find a $2$-path covering of $\mathcal{P}_n^-$ consisting of a path $P_2$ from $r_1^-$ to $r_2^-$ through the edge $r^-g^-$ and of a path $P_3$ from $r_3^-$ to $r_4^-$. Call $P_4$ the path obtained from $P_2$ by replacing the edge $r_3^-r^-$ with $r_3^-r^+r^-$. The desired Hamiltonian path is obtained from $P_1$ by replacing $r_1^+\mathbf{0}^+r_2^+$ with the concatenation of $r_1^+r_1^-$, $P_4$ and $r_2^-r_2^+$, and replacing $r_3^+\mathbf{1}^+r_4^+$ with the concatenation of $r_3^+r_3^-$, $P_3$ and $r_4^-r_4^+$.
    
        \item \textbf{Case $2d)$:} $u$ is in $\mathcal{P}_n^+$, $v$ is in $\mathcal{P}_n^-$, $u$ is green and $v$ is red. Choose a red vertex $r^+$ in $C_{n,1}^+$ and a green neighbor $g^-$ of $r^+$ in $C_n^-$. Choose an edge $e$ at $r^+$ different from $\mathbf{0}^+r^+$. By \Cref{item:3.2}, there is a Hamiltonian path $P_1$ in $\mathcal{Q}_n^+$ from $u$ to $r^+$ through $e$. Let $r_1^+, r_2^+$ be the neighbors of $\mathbf{0}^+$ along $P_1$ (with $r_1^+$ before $r_2^+$), and $r_3^+,r_4^+$ be the neighbors of $\mathbf{1}^-$ along $P_1$ (with $r_3^+$ before $r_4^+$). Let $r_1^-,r_2^-,r_3^-,r_4^-$ be the corresponding copies in $\mathcal{P}_n^-$. By \Cref{item:5.10}, there is a $3$-path covering of $\mathcal{P}_n^-$ with a path $P_2$ from $r_1^-$ to $r_2^-$, a path $P_3$ from $r_3^-$ to $r_4^-$ and a path $P_4$ from $g^-$ to $v$. Suitable substitution of parts of $P_1$ with stitching edges and the paths $P_2$, $P_3$ and $P_4$ yields the desired Hamiltonian path.
        \item \textbf{Case $2e)$:} $u$ is in $\mathcal{P}_n^+$, $v$ is in $\mathcal{P}_n^-$ and both are red. Let $r_1^+$, $r_2^+$, $r_3^+$ be three distinct lean red vertices in $\mathcal{P}_n^+$, and such that none of them is equal to $u$, and none of their corresponding copies $r_1^-$, $r_2^-$, $r_3^-$ is equal to $v$. Applying \Cref{item:3.13} on both $\mathcal{P}_n^+$ and $\mathcal{P}_n^-$ we can find $2$-path coverings $P_1$, $P_2$ of $\mathcal{P}_n^+$, $P_3$, $P_4$ of $\mathcal{P}_n^-$, such that $P_1$ connects $u$ to $r_1^+$, $P_2$ connects $r_2^+$ to $r_3^+$, $P_3$ connects $r_1^-$ to $r_2^-$ and $P_4$ connects $r_3^-$ to $v$. We obtain the desired Hamiltonian path from $P_1,P_2,P_3,P_4$ and all the stitching edges $r_i^+r_i^-$ for $i=1,2,3$.
        \item \textbf{Case $2f)$:} $u$ is in $\mathcal{P}_n^+$, $v$ is in $\mathcal{P}_n^-$ and both are green. This proof is analogous to the case (d) above, but instead of using a green neighbor $g^-$ of $r^+$, we use the red copy $r^-$ of $r^+$.
    \end{itemize}
   
\end{proof}

\begin{figure}
    \centering
    \includegraphics[width=\textwidth]{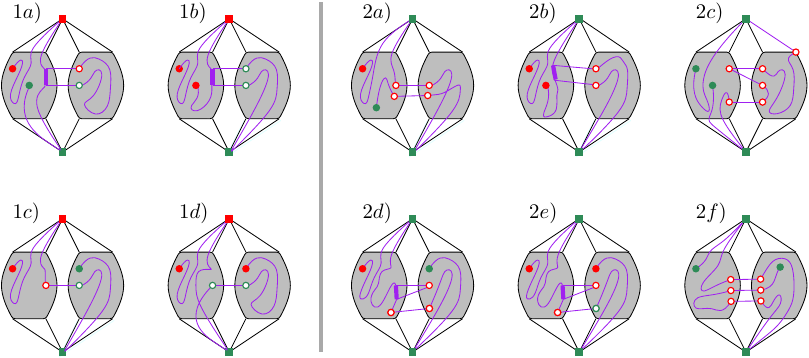}
    \caption{Cases in the proof of~\Cref{hamwhirls}. Left: Odd whirls. Right: Even whirls.}
    \label{fig:whirls}
\end{figure}

Finally, we prove Hamiltonian connectivity for base-cobase graphs of whirls.

\begin{thm}[Whirls satisfy \textsc{Ham}]\label{hamwhirls}
    Let $n\geq 3$. The base-cobase graph $H$ of $\mathcal{W}^n$ is Hamiltonian connected.
\end{thm}

\begin{proof}
    The case $n=3$ can be dealt with computationally, so we may assume that $n\geq 4$ and thus that we can use any of the auxiliary results above. Let $u$ and $v$ be distinct vertices of $H$. We split into cases in a similar way as for wheels. For the most part, our analysis will directly include the possibility that $u$ or $v$ are $\mathbf{0}$ or $\mathbf{1}$, but in Case $2c)$ below we will need to make a separate analysis to include this case as well. Also, note that the following proofs are less detailed, as most of the stitching and concatenation ideas were presented before. Refer to~\Cref{fig:whirls} for illustration.

    \textbf{Case $1$:} $n$ odd. We have the following subcases.

    \begin{itemize}
        \item \textbf{Case $1a)$:} $u,v$ lie in $\mathcal{Q}_n^+$ and within this part, they have opposite colors, say $u\in R^+$ and $v\in G^+$. We choose any lean edge $e=g^+r^+$ in $\mathcal{Q}_n^+$ disjoint from $uv$. Let $g^-$, $r^-$ be the corresponding vertices in $\mathcal{P}_n^-$. By \Cref{item:3.2}, there is a Hamiltonian path $P_1$ in $\mathcal{Q}_n^+$ from $u$ to $v$ through $e$. By  \Cref{item:3.6}, there is a Hamiltonian path $P_2$ in $\mathcal{P}_n^-$ from $r^-$ to $g^-$. The desired Hamiltonian path is obtained by replacing $e$ by the stitching edges $g^+g^-$, $r^+r^-$, and adding the path $P_2$.  
        \item \textbf{Case $1b)$:} $u,v$ lie in $\mathcal{Q}_n^+$ and within this part, they have the same color, say $u,v \in R^+$. We choose a lean edge $e=g^+r^+$ in $\mathcal{Q}_n^+-\mathbf{0}$ disjoint from $uv$. We choose distinct green neighbors $g_1^-$ and $g_2^-$ in $\mathcal{P}_n^-$ of $g^+$ and $r^+$, respectively. By \Cref{item:3.12} there is a Hamiltonian path $P_1$ of $\mathcal{Q}_n^+-\{\mathbf{0}\}$ from $u$ to $v$ through $e$. By \Cref{item:3.12} (choosing any edge), there is a Hamiltonian path $P_2$ of $\mathcal{Q}_n^--\{\mathbf{1}\}$ from $g_1^-$ to $g_2^+$. The desired path is a suitable replacement and stitching of $P_1$ and $P_2$.  
        \item \textbf{Case $1c)$:} $u$ lies in $R^+$ and $v$ in $G^-$. We choose any red lean vertex $r^+$ different from $u$ and $\mathbf{1}^+$. Let $g^-$ be a green neighbor of $r^+$ in $\mathcal{P}_n^-$. By \Cref{item:3.12} applied twice, there is a Hamiltonian path $P_1$ in $\mathcal{Q}_n^+-\mathbf{0}$ from $u$ to $r^+$ and a Hamiltonian path $P_2$ in $\mathcal{Q}_n^--\mathbf{1}$ from $g^-$ to $v$. The concatenation of $P_1$, $r^+g^-$ and $P_2$ is the desired Hamiltonian path.
        \item \textbf{Case $1d)$:} $u$ lies in $R^+$ and $v$ in $R^-$. $u$ lies in $R^+$ and $v$ in $R^-$. This case is analogous to the previous case, but using the copy $r^-$ instead of $g^-$.
    \end{itemize}

    \textbf{Case $2$:} $n$ even. 

    \begin{itemize}
        \item \textbf{Case $2a)$:} $u,v$ lie in $\mathcal{Q}_n^+$, $u$ is red and $v$ is green. Let $r_1^+$, $r_2^+$ be two distinct red lean vertices in $\mathcal{P}_n^+$ and different from $u$. Let $r_3^-$ and $r_4^-$ be red neighbors of $r_1^+$ and $r_2^+$, respectively, and in $\mathcal{P}_n^-$. By \Cref{item:3.11} there is a $2$-path covering of $\mathcal{Q}_n^+-\{\mathbf{0}\}$ with one path $P_1$ from $u$ to $r_1^+$ and the other path $P_2$ from $r_2^+$ to $v$. By \Cref{item:3.12} there is a Hamiltonian path $P_3$ from $r_3^-$ to $r_4^-$ in $\mathcal{Q}_n^--\{\mathbf{1}\}$, Suitable stitching of $P_1$, $P_2$ and $P_3$ yields the desired path.
        
        \item \textbf{Case $2b)$:} $u,v$ lie in $\mathcal{Q}_n^+$ and both are red. Let $e=r^+g^+$ be a lean edge of $\mathcal{P}_n^+$. Let $r^-$ be the corresponding copy of $r^+$ in $\mathcal{P}_n^-$ and $r_1^-$ a red neighbor of $g^+$ in $\mathcal{P}_n^-$. By \Cref{item:3.12}, there is a Hamiltonian path $P_1$ in $\mathcal{Q}_n^+-\{\mathbf{0}\}$ from $u$ to $v$ through $e$. By \Cref{item:3.12}, there is a Hamiltonian path $P_3$ in $\mathcal{Q}_n^--\{\mathbf{1}\}$ from $r^-$ to $r_1^-$. A suitable stitching of paths $P_1$ and $P_2$ yield the desired Hamiltonian path.

        \item \textbf{Case $2c)$:} $u,v$ lie in $\mathcal{Q}_n^+$ and both are green. If $u$ and $v$ are in some order $\mathbf{0}$ and $\mathbf{1}$, we connect them to arbitrary vertices $u'$ and $v'$ and find a Hamiltonian path $P_1$ between $u'$ and $v'$ in the base-cobase graph of the wheel by~\Cref{hamwheels}. Then the concatenation of $uu'$, $P_1$ and $v'v$ is the desired Hamiltonian path.

        Otherwise, we can assume by the \emph{vertical} symmetry $B_{n,k} \leftrightarrow B_{n,n-k}$ that none of $u$ and $v$ is $\mathbf{1}$. Let $r_1^+$ and $r_2^+$ be two distinct red lean vertices of $\mathcal{P}_n^+$. Let $r_1^-$ be the copy of $r^+$ in $\mathcal{P}_n^-$ and $r_2^-$, $r_3^-$ be two distinct red neighbors of $r_2^+$ in $\mathcal{P}_n^-$, and different from $r_1^-$. By \Cref{item:3.11}, there is a $2$-path covering of $\mathcal{Q}_n^+-\{r_2^+\}$ with one path $P_1$ from $u$ to $\mathbf{1}$ and the other path $P_2$ from $r_1^+$ to $v$. Now consider a vertex $r^-$ in $B_{n,n-1}^-$ (and thus neighbor of $\mathbf{1}$). By \Cref{item:4.3}, there is a $2$-path cover of $\mathcal{P}_n^-$ with one path $P_3$ from $r^-$ to $r_2^-$ and the other path $P_4$ from $r_3^-$ to $r_1^-$. The concatenation of $P_1$, $\mathbf{1}r^-$, $P_3$, $r_2^-r_2^+r_3^-$, $P_4$, $r_1^-r_1^+$, $P_2$, yields the desired Hamiltonian path.
    
        \item \textbf{Case $2d)$:} $u$ is in $\mathcal{Q}_n^+$, $v$ is in $\mathcal{Q}_n^-$, $u$ is red and $v$ is green. By the vertical symmetry, we may assume $v$ is not $\mathbf{1}$. Let $r_1^+$ be any red lean vertex of $\mathcal{P}_n^+$, and $e$ a lean edge of $\mathcal{P}_n^+$ disjoint from $\{u,r_1^+\}$. Let $r_1^-$ be the copy of $r_1^+$ in $\mathcal{P}_n^-$ and let $r_2^-$ be a red vertex different from $r_2^-$ that can be grabbed with the edge $e$. By \Cref{item:3.12}, there is a Hamiltonian path $P_1$ of $\mathcal{Q}_n^+-\{\mathbf{0}\}$ from $u$ to $r_1^+$. By \Cref{item:3.6} there is a Hamiltonian path $P_2$ of $\mathcal{Q}_n^--\{\mathbf{1},r_2^-\}$ from $r_1^-$ to $v$. The desired path is obtained using path $P_1$ (with $e$ replaced by the grabbing of $r_2^+$), the edge $r_1^+r_1^-$ and the path $P_2$. 
        
        \item \textbf{Case $2e)$:} $u$ is in $\mathcal{Q}_n^+$, $v$ is in $\mathcal{Q}_n^-$ and both are red. This case is analogous to the previous one, but using a green neighbor $g_1^-$ of $r_1^+$ instead of its copy $r_1^-$.
        
        \item \textbf{Case $2f)$:} $u$ is in $\mathcal{Q}_n^+$, $v$ is in $\mathcal{Q}_n^-$ and both are green. By vertical symmetry, we may assume that $u$ is not $\mathbf{0}$ and $v$ is not $\mathbf{1}$. Let $r_1^+,r_2^+,r_3^+$ be three distinct red lean vertices in $\mathcal{P}_n^+$. Let $r_1^-,r_2^-,r_3^-$ be their copies in $\mathcal{P}_n^-$. By \Cref{item:3.11}, there is a $2$-path covering of $\mathcal{Q}_n^+$ with a path $P_1$ from $u$ to $r_1^+$ and a path $P_2$ from $r_2^+$ and $r_3^+$. By \Cref{item:3.11}, there is a $2$-path covering of $\mathcal{Q}_n^--\{\mathbf{1}\}$ with a path $P_3$ from $r_1^-$ to $r_2^-$ and a path from $r_3^-$ to $v$. A suitable concatenation of $P_1,P_2,P_3,P_4$ and stitching edges $r_i^+r_i^-$ for $i=1,2,3$ yields the desired Hamiltonian path.
        
    \end{itemize}
    
\end{proof}

\section{Regular matroids}\label{sec:regular}

In this section we present results related to regular matroids. First, we prove that regular matroids rarely satisfy \textsc{Mat}.

\begin{thm}[\textsc{Mat} is rare in regular matroids]\label{regmat}
    Let $M$ be a regular block matroid, then the following are equivalent:
    \begin{itemize}
    \item[(i)] $M$ satisfies \textsc{Mat},
    \item[(ii)] $M$ is the graphic matroid of a series-parallel graph,
    \item[(iii)] $M=U_{1,2}\oplus\cdots\oplus U_{1,2}$.
\end{itemize}
\end{thm}
\begin{proof}
\noindent (i)$\Rightarrow$(iii)]
 Let $M=(E,\mc{B}_M)$ be regular and satisfying \textsc{Mat}. Note that the class of regular matroids is closed under minors and direct sums. We proceed by induction on $|E|$. If $M$ is disconnected then the result follows by induction and~\Cref{sums}. Otherwise, by~\Cref{selfdual} there exists an identically self-dual matroid $N$, such that $\mathcal{B}_N\subseteq \mathcal{B}_M$ and $\mathcal{B}_N\cap\mathcal{B}_{N^*}=\mathcal{B}_M\cap\mathcal{B}_{M^*}$. By~\cite{Luc75} $N\cong M/F\oplus M_F$ for some $F\subseteq E$. If $F\neq E,\emptyset$, then the result follows by induction and~\Cref{sums}. Hence, by~\Cref{selfdual} we have that $M=N$ is identically self-dual and connected. By~\cite{Lin84} we have that $M=U_{1,2}$.
 
 \noindent (iii)$\Rightarrow$(ii)] is trivial
 \noindent (ii)$\Rightarrow$(i)] follows from \Cref{thm:matroidal} since graphic matroids from series-parallel graphs are in $\spex(\mathrm{LPM})$.
\end{proof}

Now, we present results for the matroid $R_{10}$. Quoting Oxley, ``the matroid $R_{10}$ has many attractive features'' \cite{Oxley2006}. It plays a fundamental role in Seymour's Decomposition Theorem for regular matroids~\cite{Sey80}. It has several representations. For our purposes, we start with its representation as the linear matroid of the following matrix in $\mathbb{F}_2^{5\times 10}$ (see~\cite{Oxley2006}).

\begin{align*}
\begin{pmatrix}
    1 & 1 & 1 & 1 & 1 & 1 & 0 & 0 & 0 & 0 \\
    1 & 1 & 1 & 0 & 0 & 0 & 1 & 1 & 1 & 0 \\
    1 & 0 & 0 & 1 & 1 & 0 & 1 & 1 & 0 & 1 \\
    0 & 1 & 0 & 1 & 0 & 1 & 1 & 0 & 1 & 1 \\
    0 & 0 & 1 & 0 & 1 & 1 & 0 & 1 & 1 & 1 \\
\end{pmatrix},
\end{align*}

Note that the columns are the $10$ incidence vectors of the elements of $\binom{[5]}{3}$. A collection of these columns is linearly independent over $\mathbb{F}_2$ if none of its non-trivial subcollections corresponds to sets in $\binom{[5]}{3}$ that cover each element in $\{1,2,3,4,5\}$ an even number of times. So an equivalent description of $R_{10}$ is as the matroid with ground set $\binom{[5]}{3}$ and independence given by the covering restriction just given. For brevity, we use $abc$ to denote $\{a,b,c\}$. For $a\in \{1,2,3,4,5\}$ and $A\in \binom{[5]}{3}$ we say that $a$ and $A$ are \emph{incident} if $a\in A$. Given a set of triples $\mathcal{A}\subseteq \binom{[5]}{3}$ and $a\in \{1,2,3,4,5\}$, we define the \emph{degree of $a$} as $\deg(a)=|\{A\in \mathcal{A}: a\in A\}|$.

With this point of view, in this section we provide a complete combinatorial description of the base-cobase graph $G_{R_{10},R_{10}^\ast}$ that in particular shows that it is bipartite, and hence not Hamiltonian connected. It is known that $R_{10}$ only has circuits of size $4$ or $6$ (see \cite{Oxley2006}). The following proposition provides a proof of this fact and describes the circuits under our notation.

\begin{prop}
\label{prop:r10circ}
    Any circuit of $R_{10}$ is either of size $4$ or $6$. If it is of size $4$ then it consists of the elements $abc,abe,acd,ade$ for some permutation $abcde$ of $12345$. If it is of size $6$, then it is the complement of a circuit of size $4$.
\end{prop}

\begin{proof}
    A circuit $C$ must consist of triples for which every element has even degree. If $|C|=k$, then we have $3k$ incidences. By double counting, $3k$ is also the total degree sum, so it must be even. Then $k$ is even as well. Since the rank of $R_{10}$ is $5$, we have $k\leq 6$. Then $k=2,4,6$.
    Having only two triples yields elements of degree $1$, so $|C|\neq 2$.
    
    Suppose $|C|=4$. If for some $a$ we have $\text{deg}(a)=0$, then $C$ consists of the four triples of $\{1,2,3,4,5\}\setminus \{a\}$, which yields four elements of odd degree $3$, a contradiction. So each degree, being an even number, is either $2$ or $4$. If $r,s$ elements have degree $2,4$ respectively, we have $r+s=5$ and $2r+4s=12$, therefore $s=1$ and $r=4$. Then there is an element $a$ of degree $2$ and the rest of degree $4$. The restriction of the triples to $\{1,2,3,4,5\}\setminus\{a\}$ is a degree $2$ graph, so it must be a union of cycles. Having $4$ elements, it must be a $4$-cycle, as desired.

    Now, suppose $|C|=6$. The number of incidences, and thus the total degree sum is $18$. If $\deg(a)=6$ for some element $a$, then $C$ are the six triples at $a$, and thus the rest of the elements have degree $3$, a contradiction. If $\deg(a)=0$, then we can have at most $4$ triples, a contradiction. So again, let $r,s$ be the elements of degree $2,4$ respectively. We have $r+s=5$, $2r+4s=18$. From here $r=1$ and $s=4$. If $a$ is the element of degree $2$, then the remaining four triples must be the ones from $\{1,2,3,4,5\}\setminus\{a\}$. These last triples cover each element $3$ times, so the two triangles at $a$ must be of the form $abc$, $ade$. The circuit is then $abc$, $ade$, $bcd$, $bce$, $bde$, $cde$.

    Given the structures above, a routine check shows that indeed the complement of a circuit of size $4$ is a circuit of size $6$, and viceversa.
\end{proof}

We now focus on describing the base-cobases of $R_{10}$.

\begin{prop}
\label{prop:r10bc}
    The base-cobases of $R_{10}$ are of exactly two types:
    \begin{itemize}
        \item For a permutation $abcde$ of $12345$, take the triples $abd$,  $acd$, $bcd$, $bce$, $bde$. 
        \item For a permutation $abcde$ of $12345$, take the triples $abd$, $acd$, $ace$, $bce$, $bde$.
    \end{itemize}
\end{prop}

Before proving the result, note that in the first case the reversing or any cyclic permutation of $abcde$ yields the same triples. Similarly, in the second case the reversing of the elements yields the same triples. We now focus on proving \Cref{prop:r10bc} and after that we discuss these symmetries in detail.

\begin{proof}
    Let $B$ be a base-cobase, so $\overline{B}$ is a base as well. Each base consists of $5$ triples, totaling $15$ incidences, and therefore the total degree sum is $15$ as well. If $\deg(a)\geq 5$ for some $a$, then the triples would induce a $5$ edge graph on $\{1,2,3,4,5\}\setminus \{a\}$, which will have a $4$-cycle and by \Cref{prop:r10circ} then $B$ would contain a circuit, a contradiction, so $\deg(a)\leq 4$. An analogous proof using $\overline{B}$ shows $\deg(a)\geq 2$. Therefore, each element has degree $2$, $3$ or $4$.

    If some element, say $b$, has degree $4$, the graph on $\{1,2,3,4,5\}\setminus\{b\}$ induced by the triples at $b$ has $4$ edges, but must not be a $4$-cycle, so it is a triangle and an intersecting edge. Without loss of generality, the triples at $b$ are then $abd, bcd, bce, bde$. Adding the triple $ace$ leads to a contradiction as then the triples $abc$, $abe$, $acd$, $ade$ would be a circuit in $\overline{B}$. So we may only add either the triple $ade$ or the triple $acd$. Both cases are symmetric. If we add the triple $acd$ we get the desired structure. Note that the degrees of $12345$ would respectively be $24342$ so we cannot remove a triangle to get the degrees $42222$ of a circuit. The same is true for $\overline{B}$, verifying that $B$ is a base-cobase.

    If no element has degree $4$, then each degree is $2$ or $3$. But since total degree sum is $15$, each degree must be $3$. For an element $a$, the two triples not at $a$ are two triples in $\{1,2,3,4,5\}\setminus\{a\}$ so they must be two triples that share an edge, say, $bce$ and $bde$. If we do not use the triple $acd$ for $B$, we need two more triples for $d$ to get degree $3$, and the same is true for $c$, yielding $6$ triples, a contradiction. Then $acd$ must also be a triple of $B$. The remaining triple at $d$ must be $abd$ or $ade$. By symmetry, we may assume it is $abd$. Then the last triple is $ace$, yielding the desired structure. With this structure, the degrees of $12345$ (for both $B$ and $\overline{B}$) are $33333$, so no triple can be removed to get the incidences $42222$ of a circuit. This shows that $B$ is indeed a base-cobase.
\end{proof}

We now discuss the symmetries involved in the description of the base-cobases in \Cref{prop:r10bc}. In the first case, the reversal $abcde\to edcba$ yields a permutation with the same triples, and these are the only symmetries. Therefore, we may identify each of these base-cobases with an element $[abcde]_{S_2}$ in $S_5/S_2$, where $S_5$ and $S_2$ are the permutation groups in $5$ and $2$ elements respectively\footnote{We use $G/H$ in the sense of the orbits of the action of $G$ under $H$, not in the sense of quotient groups.}. Similarly, in the second case the reversal $abcde\to edcba$ and the cyclic shift $abcde\to bcdea$ yield a permutation with the same triples, and these are the only symmetries. Thus, we can identify each base-cobase with an element $[abcde]_{D_5}$ of $S_5/D_5$, where $D_5$ is the dihedral group on $5$ elements. In particular, there are $5!/2=60$ base-cobases of the first kind and $5!/10=12$ of the second kind.

We are ready to provide the main result of this section, the complete description of $G_{R_{10},R_{10}^\ast}$.

\begin{thm}
\label{thm:r10desc}
    The graph $G_{R_{10},R_{10}^\ast}$ is isomorphic to the graph with vertex set $(S_5/D_5) \cup (S_5/S_2)$ where:
    \begin{itemize}
        \item The neighbors of $[abcde]_{S_2}$ are $[abcde]_{D_5}$,  $[cbade]_{S_2}$, $[abcde]_{S_2}$, $[abdce]_{S_2}$, $[abedc]_{S_2}$. 
        \item The neighbors of $[abcde]_{D_5}$ are $[abcde]_{S_2}$, $[bcdea]_{S_2}$, $[cdeab]_{S_2}$, $[deabc]_{S_2}$ and $[eabcd]_{S_2}$.
    \end{itemize}
\end{thm}

\begin{proof}
    By \Cref{prop:r10bc} and the symmetry discussion above, each base-cobase of the first kind can be identified with an element $[abcde]_{S_2}$ of $S_5/S_2$ and each of the second kind with an element of $[abcde]_{D_5}$ of $S_5/D_5$, so the vertices are correct. Remember that an edge in the base-cobase graph is between base-cobases with symmetric difference $2$, so we have to obtain one from the other by removing and adding a triple.

    Consider a base-cobase $B=[abcde]_{S_2}$. Its triples are $abd$, $acd$, $bcd$, $bce$, $bde$. The degrees of vertices $abcde$ are $24342$. We have the following cases:
    
    \begin{itemize}
        \item If we remove the triple $bcd$, the degree sequence becomes $23232$. To get degree sequence $33333$ we can only add the triangle $ace$ and we obtain $[abcde]_{D_5}$. To get degree sequence $24342$ (or a permutation), the $3$'s are the only ones that can become $4$'s, so the only possible triangles to add are $abd$, $bcd$ or $bde$, but we already have them. So there are no more neighbors in this case.
        \item If we remove the triangle $abd$ the degree sequence becomes $13332$. Adding a triangle yields an element of degree $4$ so there are no neighbors in $S_5/D_5$ in this case. To get a sequence $24342$ (or a permutation), the triple must have $a$ (to avoid an element of degree $1$), and two of $b,c,d$ (to get the two elements of degree $4$). So the triple is either $abc$, $abd$ or $acd$. Adding $abd$ is going back, $acd$ is already in $B$, so the only possible case is adding $abc$, which yields $[abdce]_{S_2}$,
        \item The triangle $bde$ is symmetric and yields the neighbor $[acbde]_{S_2}$.
        \item If we remove the triangle $bce$, we get degree sequence $23241$. We cannot get rid of the element of degree $4$, so in this case there are no neighbors in $S_5/D_5$. To get a sequence $24342$ (or a permutation), we must use element $e$ and cannot use element $d$. Then the triple to add is either $abe$, $ace$ or $bce$. The triple $bce$ goes back and the triple $ace$ yields degree sequence $33342$, not corresponding to a base-cobase. So the triple to add must be $abe$, which yields $[cbade]_{S_2}$.
        \item The triangle $abd$ is symmetric and yields the neighbor $[abedc]_{S_2}$.
    \end{itemize}
    
    Now consider a base-cobase $[abcde]_{D_5}$. The analysis above shows that its neighbors in $S_5/S_2$ are $[abcde]_{S_2}$, $[bcdea]_{S_2}$, $[cdeab]_{S_2}$, $[deabc]_{S_2}$ and $[eabcd]_{S_2}$. It cannot be connected to a vertex in $S_5/D_5$ since its degree sequence is $33333$, so by removing a triangle the only way to recover the same degree sequence is by putting the same triangle back. 
\end{proof}

As as corollary we obtain the following.

\begin{cor}
\label{prop:errediez}
The base-cobase graph of $R_{10}$ is bipartite and therefore not Hamiltonian connected.
\end{cor}

\begin{proof}
Each permutation $abcde$ in $S_5$ has a sign, which is preserved by reversal and cyclic shifts. Thus, the sign of a class $[abcde]_{S_2}$ and $[abcde]_{D_5}$ is well-defined.

We propose the following partition of $G_{R_{10},R_{10}^\ast}$:

\begin{itemize}
    \item $X$ are the vertices $[abcde]_{D_5}$ with odd sign, and the vertices $[abcde]_{S_2}$ with even sign.
    \item $Y$ are the vertices $[abcde]_{S_2}$ with even sign, and the vertices $[abcde]_{S_2}$ with odd sign. 
\end{itemize}

By~\Cref{thm:r10desc}, an element $[abcde]_{D_5}$ connects only with the elements with same sign (and thus in opposite side of the partition) $[abcde]_{S_2}$, $[bcdea]_{S_2}$, $[cdeab]_{S_2}$, $[deabc]_{S_2}$ and $[eabcd]_{S_2}$.

An element $[abcde]_{S_2}$ connects with the vertex with the same sign (and opposite side of the partition) $[abcde]_{D_5}$. Its other neighbors are $[cbade]_{S_2}$, $[abcde]_{S_2}$, $[abdce]_{S_2}$, $[abedc]_{S_2}$, which are one transposition away. Therefore they have opposite sign, so they are also in the opposite side of the partition. This proves that the graph is bipartite.

\end{proof}

We remark that with the help of a computer, we could verify that the base-cobase graph of $R_{10}$ is Hamiltonian laceable. The above raises the following question.

\begin{quest}
    Are the bipartiton classes of any bipartite base-cobase graph equicardinal?    
\end{quest}

Note that a negative answer would in particular yield a non-Hamiltonian base-cobase graph. 




\section*{Acknowledgements}
\addcontentsline{toc}{section}{Acknowledgements}

We thank Carolina Benedetti for discussions on positroids, Micha{\l} Lasoń for explanations concerning White's conjectures, and Jorge Ramirez-Alfonsín for remarks on LPM.
K.K. was supported by the Spanish State Research Agency
through grant PID2022-137283NB-C22 and the Severo Ochoa and María de Maeztu Program for Centers and Units of Excellence in R\&D (CEX2020-001084-M).

\section{Appendix: Whirls as transversal matroids}
\label{sec:whirlstransv}
\addcontentsline{toc}{section}{Appendix A: Whirls as transversal matroids}

\begin{prop}
    The matroids $N_n$ and $\mathcal{W}^n$ are isomorphic.
\end{prop}

\begin{proof}
    The proposed isomorphism is sending the edge $i$ in $\mathcal{W}^n$ to the element $i$ in $N_n$. We show the isomorphism via bases.

    The outer cycle of $\mathcal{W}^n$ maps to the elements $1,3,5,\ldots,2n-1$ of $N_n$, which clearly is a transversal for $I_1,\ldots,I_n$. Any other base of $\mathcal{W}^n$ is a spanning tree $T$ of $\mathcal{W}_n$. To show that the elements corresponding to $T$ yield a transversal in $N_n$, it is enough by Hall's matching theorem~\cite{Hal35} to show that in the union of any $k$ intervals from $I_1,\ldots,I_n$ there are at least $k$ elements of $T$. This is certainly true for $k=n$, since then the union is the whole graph, with $n+1$ vertices and exactly $n$ edges from $T$.

    Now, each $I_i$ is a claw centered at vertex $v_{i+1}$. Any union of $k$ of them may be split into blocks corresponding to claws centered at consecutive vertices. We show inductively that a block corresponding to $\ell<n$ vertices has at least $\ell$ edges from $T$. Such a block has $\ell+2$ vertices and $2\ell+1$ edges.
    
    For $\ell=1$, $T$ has an edge incident to the center of the claw. Suppose the result is true for a block for $\ell-1$ consecutive vertices and consider a block corresponding to $\ell$ consecutive vertices $w_1,w_2,\ldots,w_\ell$. Let $w_0$ be the vertex before $w_0$ and $w_{\ell+2}$ the vertex after $w_\ell$. If any of the edges $w_0w_1$, $cw_1$, $w_{\ell}c$ or $w_{\ell}w_{\ell+1}$ are in $T$, we may remove the first or last claw and apply the inductive hypothesis to get $\ell-1$ edges of $T$ and return the claw to get $\ell$ edges. The remaining case is that none of those edges are in $T$. But then $T$ must connect the vertices $w_1,\ldots,w_\ell$ to the rest of the graph via $c$, so the restriction of $T$ to $w_1,\ldots w_\ell,c$ must be spanning. Therefore, this restriction has at least $\ell$ edges, as desired. This finishes the inductive step.

    Now suppose that $n$ elements from $\{1,2,\ldots,2n\}$ are a transversal for $I_1,\ldots,I_n$. If these elements do not correspond to the edges of a spanning tree of $\mathcal{W}_n$, it is because these edges form a cycle. If this is the outer cycle, then there is no problem, as this is a base of $\mathcal{W}^n$. Otherwise, we may assume this cycle is minimal and without loss of generality to be $v_1,v_2,\ldots,v_l,c$, using then $\ell+1$ edges. But since the chosen elements are a transversal, then in the claws centered at $v_{\ell+1},\ldots,v_n$ we must have at least some additional $n-\ell$ edges. This yields at least $(\ell+1)+(n-\ell)=n+1$ edges, contradicting the fact that we originally chose $n$. 
\end{proof}

\small
\bibliography{lpbib}

\def\cprime{$'$}
\begin{thebibliography}{10}

\bibitem{An-17}
{\sc S.~An, J.~Jung, and S.~Kim}, {\em Facial structures of lattice path
  matroid polytopes}, Discrete Math., 343 (2020), p.~11.
\newblock Id/No 111628.

\bibitem{AHM14}
{\sc S.~D. Andres, W.~Hochst{\"a}ttler, and M.~Merkel}, {\em On a base exchange
  game on bispanning graphs}, Discrete Appl. Math., 165 (2014), pp.~25--36.

\bibitem{benedetti2025shellabilityquotientorderlattice}
{\sc C.~Benedetti, A.~Dochtermann, K.~Knauer, and Y.~Li}, {\em Shellability of
  the quotient order on lattice path matroids}, arXiv preprint
  arXiv:2504.07306,  (2025).

\bibitem{benedetti2023latticepathmatroidpolytopes}
{\sc C.~Benedetti, K.~Knauer, and J.~Valencia-Porras}, {\em On lattice path
  matroid polytopes: alcoved triangulations and snake decompositions}, arXiv
  preprint arXiv:2303.10458,  (2023).

\bibitem{BK22}
{\sc C.~Benedetti-Vel{\'a}squez and K.~Knauer}, {\em Lattice path matroids and
  quotients}, Combinatorica, 44 (2024), pp.~621--650.

\bibitem{brczi2023reconfiguration}
{\sc K.~B{\'e}rczi, B.~M{\'a}trav{\"o}lgyi, and T.~Schwarcz}, {\em
  Reconfiguration of basis pairs in regular matroids}, in STOC 2024:
  Proceedings of the 56th Annual ACM Symposium on Theory of Computing,
  Association for Computing Machinery, 2024, pp.~1653--1664.

\bibitem{BMS24}
{\sc K.~B{\'e}rczi, B.~M{\'a}trav{\"o}lgyi, and T.~Schwarcz}, {\em Weighted
  exchange distance of basis pairs}, Discrete Appl. Math., 349 (2024),
  pp.~130--143.

\bibitem{Berczi2022}
{\sc K.~B{\'e}rczi and T.~Schwarcz}, {\em Exchange distance of basis pairs in
  split matroids}, SIAM J. Discrete Math., 38 (2024), pp.~132--147.

\bibitem{B07}
{\sc S.~{Blum}}, {\em {Base-sortable matroids and Koszulness of semigroup
  rings}}, {Eur. J. Comb.}, 22 (2001), pp.~937--951.

\bibitem{Bon72}
{\sc J.~A. Bondy}, {\em Transversal matroids, base-orderable matroids, and
  graphs}, Q. J. Math., Oxf. II. Ser., 23 (1972), pp.~81--89.

\bibitem{Bon2003}
{\sc J.~Bonin, A.~{de Mier}, and M.~Noy}, {\em Lattice path matroids:
  enumerative aspects and {T}utte polynomials}, J. Combin. Theory Ser. A, 104
  (2003), pp.~63--94.

\bibitem{Bon-10}
{\sc J.~E. Bonin}, {\em Lattice path matroids: the excluded minors}, J. Combin.
  Theory Ser. B, 100 (2010), pp.~585--599.

\bibitem{Bonin2013}
{\sc J.~E. Bonin}, {\em Basis-exchange properties of sparse paving matroids},
  Adv. Appl. Math., 50 (2013), pp.~6--15.

\bibitem{Bon-03}
{\sc J.~E. Bonin, A.~de~Mier, and M.~Noy}, {\em Lattice path matroids:
  enumerative aspects and {T}utte polynomials}, J. Combin. Theory Ser. A, 104
  (2003), pp.~63--94.

\bibitem{Bon-07}
{\sc J.~E. Bonin and O.~Gim{\'e}nez}, {\em Multi-path matroids}, Combin.
  Probab. Comput., 16 (2007), pp.~193--217.

\bibitem{castaneda2015path}
{\sc N.~Casta{\~n}eda and I.~S. Gotchev}, {\em Path coverings with prescribed
  ends in faulty hypercubes}, Graphs Combin., 31 (2015), pp.~833--869.

\bibitem{CCO15}
{\sc J.~Chalopin, V.~Chepoi, and D.~Osajda}, {\em On two conjectures of
  {M}aurer concerning basis graphs of matroids}, J. Comb. Theory Ser. B, 114
  (2015), pp.~1--32.

\bibitem{Chaourar2003}
{\sc B.~Chaourar and J.~Oxley}, {\em On series-parallel extensions of uniform
  matroids}, Eur. J. Comb., 24 (2003), p.~877.

\bibitem{Che17}
{\sc V.~Chepoi}, {\em Distance-preserving subgraphs of {Johnson} graphs},
  Combinatorica, 37 (2017), pp.~1039--1055.

\bibitem{CH24}
{\sc L.~Chidiac and W.~Hochst{\"a}ttler}, {\em Positroids are 3-colorable},
  Stud. Sci. Math. Hung., 61 (2024), pp.~147--160.

\bibitem{Cordovil1993}
{\sc R.~Cordovil and M.~Moreira}, {\em Bases-cobases graphs and polytopes of
  matroids}, Combinatorica, 13 (1993), pp.~157--165.

\bibitem{dS87}
{\sc I.~P. da~Silva}, ed., {\em {Quelques propri\'et\'es des matroides
  orient\'es}}, Ph.D. Dissertation, Universit\'e Paris VI, 1987.

\bibitem{DeM-07}
{\sc A.~{De Mier}}, {\em {A natural family of flag matroids}}, {SIAM J.
  Discrete Math.}, 21 (2007), pp.~130--140.

\bibitem{Edm70}
{\sc J.~Edmonds}, {\em Submodular functions, matroids, and certain polyhedra}.
\newblock Combinat. {Struct}. {Appl}., {Proc}. {Calgary} internat. {Conf}.
  combinat. {Struct}. {Appl}., {Calgary} 1969, 69-87 (1970)., 1970.

\bibitem{Far-89}
{\sc M.~Farber}, {\em Basis pair graphs of transversal matroids are connected},
  Discrete Math., 73 (1989), pp.~245--248.

\bibitem{Farber1985}
{\sc M.~Farber, B.~Richter, and H.~Shank}, {\em Edge-disjoint spanning trees: A
  connectedness theorem}, J. Graph Theory, 9 (1985), pp.~319--324.

\bibitem{FS05}
{\sc E.~M. Feichtner and B.~Sturmfels}, {\em Matroid polytopes, nested sets and
  {Bergman} fans}, Port. Math. (N.S.), 62 (2005), pp.~437--468.

\bibitem{Gab76}
{\sc H.~Gabow}, {\em Decomposing symmetric exchanges in matroid bases}, Math.
  Program., 10 (1976), pp.~271--276.

\bibitem{geiger2022selfdual}
{\sc A.~Geiger, S.~Hashimoto, B.~Sturmfels, and R.~Vlad}, {\em Self-dual
  matroids from canonical curves}, Exp. Math., 33 (2024), pp.~701--722.

\bibitem{geiger2023graph}
{\sc A.~Geiger, K.~Kuehn, and R.~Vlad}, {\em Graph curve matroids}, arXiv
  preprint arXiv:2311.08332,  (2023).

\bibitem{GGMS87}
{\sc I.~M. Gel'fand, R.~M. Goresky, R.~D. MacPherson, and V.~V. Serganova},
  {\em Combinatorial geometries, convex polyhedra, and {Schubert} cells}, Adv.
  Math., 63 (1987), pp.~301--316.

\bibitem{GMM23bis}
{\sc P.~Gregor, A.~Merino, and T.~M{\"u}tze}, {\em Star transposition {Gray}
  codes for multiset permutations}, J. Graph Theory, 103 (2023), pp.~212--270.

\bibitem{GMM23}
{\sc P.~Gregor, O.~Mi{\v{c}}ka, and T.~M{\"u}tze}, {\em On the central levels
  problem}, J. Comb. Theory, Ser. B, 160 (2023), pp.~163--205.

\bibitem{Hal35}
{\sc P.~Hall}, {\em On representatives of subsets}, J. Lond. Math. Soc., 10
  (1935), pp.~26--30.

\bibitem{havel1984Hamiltonian}
{\sc I.~Havel}, {\em On {H}amiltonian circuits and spanning trees of
  hypercubes}, {\v{C}}asopis pro p{\v{e}}stov{\'a}n{\'\i} matematiky, 109
  (1984), pp.~135--152.

\bibitem{HNT73}
{\sc C.~A. Holzmann, P.~G. Norton, and M.~D. Tobey}, {\em A graphical
  representation of matroids}, SIAM J. Appl. Math., 25 (1973), pp.~618--627.

\bibitem{Huh18}
{\sc J.~Huh}, {\em Tropical geometry of matroids}, in Current developments in
  mathematics 2016. Papers based on selected lectures given at the conference,
  Harvard University, Cambridge, MA, USA, November 2016, Somerville, MA:
  International Press, 2018, pp.~1--46.

\bibitem{Kajitani1988}
{\sc Y.~Kajitani, S.~Ueno, and H.~Miyano}, {\em Ordering of the elements of a
  matroid such that its consecutive w elements are independent}, Discrete
  Math., 72 (1988), pp.~187--194.

\bibitem{knauer2018lattice}
{\sc K.~Knauer, L.~Mart{\'\i}nez-Sandoval, and J.~L.
  Ram{\'\i}rez~Alfons{\'\i}n}, {\em On lattice path matroid polytopes: integer
  points and {E}hrhart polynomial}, Discrete Comput. Geom., 60 (2018),
  pp.~698--719.

\bibitem{KMR18}
{\sc K.~Knauer, L.~Mart{\'{\i}}nez-Sandoval, and J.~L.
  Ram{\'{\i}}rez~Alfons{\'{\i}}n}, {\em A {Tutte} polynomial inequality for
  lattice path matroids}, Adv. Appl. Math., 94 (2018), pp.~23--38.

\bibitem{LP07}
{\sc T.~Lam and A.~Postnikov}, {\em Alcoved polytopes. {I}.}, Discrete Comput.
  Geom., 38 (2007), pp.~453--478.

\bibitem{Lin84}
{\sc B.~Lindstr{\"o}m}, {\em On binary identically self-dual matroids}, Eur. J.
  Comb., 5 (1984), pp.~55--58.

\bibitem{Luc75}
{\sc D.~Lucas}, {\em Weak maps of combinatorial geometries}, Trans. Am. Math.
  Soc., 206 (1975), pp.~247--279.

\bibitem{Mat77}
{\sc L.~R. Matthews}, {\em Bicircular matroids}, Q. J. Math., Oxf. II. Ser., 28
  (1977), pp.~213--227.

\bibitem{Mau73}
{\sc S.~B. Maurer}, {\em Matroid basis graphs. {I}}, J. Comb. Theory Ser. B, 14
  (1973), pp.~216--240.

\bibitem{McG20}
{\sc S.~McGuinness}, {\em Frame matroids, toric ideals, and a conjecture of
  {White}}, Adv. Appl. Math., 118 (2020), p.~46.
\newblock Id/No 102042.

\bibitem{Mor-13}
{\sc J.~Morton and J.~Turner}, {\em Computing the {T}utte polynomial of lattice
  path matroids using determinantal circuits}, Theor. Comput. Sci., 598 (2015),
  pp.~150--156.

\bibitem{Mut16}
{\sc T.~M{\"u}tze}, {\em Proof of the middle levels conjecture}, Proc. Lond.
  Math. Soc. (3), 112 (2016), pp.~677--713.

\bibitem{Mut23}
{\sc T.~M{\"u}tze}, {\em Combinatorial {Gray} codes -- an updated survey},
  Electron. J. Comb., DS26 (2023), p.~93.

\bibitem{Mut24}
{\sc T.~M{\"u}tze}, {\em A book proof of the middle levels theorem},
  Combinatorica, 44 (2024), pp.~205--208.

\bibitem{NO84}
{\sc D.~J. Naddef and W.~R. Pulleyblank}, {\em Hamiltonicity in
  (0-1)-polyhedra}, J. Comb. Theory Ser. B, 37 (1984), pp.~41--52.

\bibitem{Oxley2006}
{\sc J.~Oxley}, {\em Matroid Theory}, Oxford graduate texts in mathematics,
  Oxford University Press, 2006.

\bibitem{Perrott2017}
{\sc A.~Perrott}, {\em Identically self-dual matroids}, master's thesis, Open
  Access Te Herenga Waka-Victoria University of Wellington, 2017.

\bibitem{P06}
{\sc A.~Postnikov}, {\em {Total positivity, Grassmannians, and networks}},
  {arXiv:0609764},  (2006).

\bibitem{Sch03}
{\sc A.~Schrijver}, {\em Combinatorial optimization. {Polyhedra} and
  efficiency}, vol.~24 of Algorithms Comb., Berlin: Springer, 2003.

\bibitem{Sch-11}
{\sc J.~Schweig}, {\em Toric ideals of lattice path matroids and polymatroids},
  J. Pure Appl. Algebra, 215 (2011), pp.~2660--2665.

\bibitem{Sey80}
{\sc P.~D. Seymour}, {\em Decomposition of regular matroids}, J. Comb. Theory
  Ser. B, 28 (1980), pp.~305--359.

\bibitem{Zie95}
{\sc G.~M. Ziegler}, {\em Lectures on polytopes}, vol.~152 of Grad. Texts
  Math., Berlin: Springer-Verlag, 1995.

\end{thebibliography}
\bibliographystyle{my-siam}

\end{document}